\documentclass[12pt]{amsart}
\usepackage{amsmath,amsthm,amssymb, enumerate}
            \newcommand{\marginalnote}[1]{}
\usepackage{hyperref}
\usepackage{marginnote}
\usepackage[all]{xy}
\usepackage{graphicx}
\usepackage{mathrsfs}
\usepackage{epsfig}
\usepackage{color}	
\usepackage{ulem}	
\usepackage{caption} 
\usepackage{mathtools}
\theoremstyle{plain}

\newtheorem{Thm}{Theorem}[section]

\newtheorem{Lem}[Thm]{Lemma}
\newtheorem{Cor}[Thm]{Corollary}

\newtheorem*{Thm*}{Theorem}
\theoremstyle{remark}

\theoremstyle{definition}

\newtheorem{Ex}[Thm]{Example}

\newtheorem{Que}[Thm]{Question}

\DeclarePairedDelimiter\floor{\lfloor}{\rfloor}

\usepackage{epigraph}

\usepackage{etoolbox}

\makeatletter
\patchcmd{\epigraph}{\@epitext{#1}}{\itshape\@epitext{#1}}{}{}
\makeatother

\begin{document}
\title{Free subgroups of free products and combinatorial hypermaps}
\author{Laura Ciobanu \& Alexander Kolpakov}
\date{\today}

\begin{abstract}
We derive a generating series for the number of free subgroups of finite index in $\Delta^+ = \mathbb{Z}_p*\mathbb{Z}_q$ by using a connection between free subgroups of $\Delta^+$ and certain hypermaps (also known as ribbon graphs or ``fat'' graphs), and show that this generating series is transcendental. We provide non-linear recurrence relations for the above numbers based on differential equations that are part of the Riccati hierarchy.

We also study the generating series for conjugacy classes of free subgroups of finite index in $\Delta^+$, which correspond to isomorphism classes of hypermaps. Asymptotic formulas are provided for the numbers of free subgroups of given finite index, conjugacy classes of such subgroups, or, equivalently, various types of hypermaps and their isomorphism classes. 

\medskip

\noindent 2010 Mathematics Subject Classification: 14N10, 20E07, 20H10, 05E45, 33C20.

\medskip

\noindent Key words: subgroup growth, growth series, free product, free group, hypermap, ribbon graph.
\end{abstract}

\maketitle

\section{Introduction}

In this paper we investigate the connection between the number of free subgroups of finite index in the free product of finite cyclic groups $\Delta^+ = \mathbb{Z}_p* \mathbb{Z}_q$, where $p, q \geq 2$ (not necessarily prime or co-prime) and the number of certain hypermaps \footnote{also known as ribbon graphs or ``fat'' graphs}, to reveal new qualitative and quantitative information about these numbers and their generating functions. We also investigate the link between conjugacy classes of free subgroups of finite index in $\Delta^+$ and isomorphism classes of hypermaps. Our present contribution is to provide new formulas, recurrence relations and asymptotics for said numbers, accompanied by software which supplies a wide spectrum of examples, and to establish the transcendence and non-holonomy of some of the associated generating functions.

The notation $\Delta^+$ is motivated by the fact that for $pq \geq 6$ the group in question is isomorphic to the $(p,q,\infty)$ orientation-preserving Fuchsian triangle group acting on two-dimensional hyperbolic space, a group whose relationship with hypermaps has been fruitfully investigated in many papers, of which we mention the ground-laying work by Jones and Singerman \cite{Jones-Singerman}, and a recent series of results by Breda d'Azevedo -- Mednykh -- Nedela \cite{Breda-Mednykh-Nedela}, and Mednykh -- Nedela \cite{MN, MN1, MN2} who solved Tutte's classification problem for maps and hypermaps, as well as the papers \cite{GIR, Imrich, Petitot-Vidal, Stothers-modular, Vidal, Yao}. 

Our methods also provide a solution to Tutte's classification problem as described in \cite{Breda-Mednykh-Nedela}, without considering the genus-specific case, c.f. \cite{MN1, MN2}. 

General subgroup growth is the subject of the monograph \cite{LS} by Lubotzky and Segal, and further information on counting the number of subgroups in free products of cyclic groups of prime orders can be found in the papers by  M\"{u}ller and Schlage-Puchta \cite{MSP1, MSP2, MSP3}. There they employ the general theory of subgroup structure in free products of (finite and infinite) cyclic groups together with representation theory, analytic number theory and probability theory, among other tools.

We use \textit{species theory} (initiated by Joyal \cite{Joyal}, c.f. also the monographs \cite{Bergeron-et-al, Flajolet-Sedgewick}) as our main computational tool, in order to generalise and enhance the results of \cite{Petitot-Vidal, Stothers-modular, Stothers}. This technique allows us to write down the generating series for the numbers of free subgroups of finite index in Theorem \ref{thm:subgroup-growth-transcendental} (or rooted hypermaps in Theorem \ref{thm:p-q-hypermaps-growth-transcendental}) and the number of their conjugacy classes in Theorem \ref{thm:conjugacy-classes} (or isomorphism classes of hypermaps in Theorem \ref{thm:p-q-hypermap-isom}) in a relatively simple form suitable for routine calculation and computer experiments. Species theory is virtually unknown in the group-theory community, and the authors hope that the applications to subgroup growth in this paper will make its use more widespread.

Another goal of the paper is to display a hierarchy of differential equations that can be used to obtain certain non-linear recurrence relations for the numbers of finite index free subgroups in $\Delta^+$. This hierarchy is known as the Riccati hierarchy c.f. \cite{GdL}, and its applications here appear as a simplified version of the general phenomenon described primarily in \cite{Goulden-Jackson, Okounkov}. Some of our generating series are associated with the classical Riccati equation, which, in particular, implies that they are non-holonomic (see Corollary \ref{cor:non-holonomic}). More connections between the classical Riccati equation, map enumeration and continued fractions representation of the respective generating series appear in \cite{AB}. 

Throughout the paper we give concrete formulas for several particular cases of free products of cyclic groups, as well as for the related hypermaps as combinatorial objects, and a SAGE code is provided in the Appendix to support our findings and to provide illustrative examples where necessary. 

\begin{small}
\section*{Acknowledgements}

The authors gratefully acknowledge the support received from the University of Neuch\^atel Overhead grant no. 12.8/U.01851. L.C. was also supported by Professorship FN PP00P2-144681/1, and A.K. by Professorship FN PP00P2-170560. A.K. is thankful to the American Institute of Mathematics (AIM) for their support during the years 2015-2017, and to the participants of the AIM SQuaRE project ``Hyperbolic Geometry beyond Dimension Three'' for fruitful discussions of this work. Both authors thank the anonymous referees for their careful reading of the manuscript and numerous helpful remarks. 
\end{small}

\section{Preliminaries}
\subsection{Hypermaps and $(p,q)$-hypermaps }\label{s:hypermaps}

Let $D = [n]$ be a set of $n$ elements, called \textit{darts}, and let $\mathfrak{S}_n$ be the corresponding symmetric group acting on them.\footnote{here and below we assume that permutations are multiplied from left to right} Let $\alpha, \sigma \in \mathfrak{S}_n$ be two permutations such that the group $\langle \alpha, \sigma \rangle$ acts transitively on $D$. The triple $H = \langle D; \alpha, \sigma \rangle$ is called an \textit{oriented labelled (combinatorial) hypermap} with set of darts $D$. The convention in the literature is to call the orbits of $\sigma$ vertices ($\sigma$ stands for the French \textit{sommets}), the orbits of $\alpha$ hyper-edges ($\alpha$ for \textit{ar\^{e}tes}) and the orbits of $\varphi = \sigma^{-1}\, \alpha^{-1}$ hyper-faces ($\varphi$ for \textit{faces}), even though these are not all technically vertices, edges and faces. The hypermap $H$ can be equivalently represented as $H = \langle D; \alpha, \varphi \rangle$. If $\alpha$ is an involution then $H$ is called a \textit{map} (see \cite{NeBook} for an introduction to maps and hypermaps). 

One might find it helpful to construct a graph where we associate with each distinct cycle of $\varphi$ a white vertex and with each distinct cycle of $\alpha$ a black vertex; a white vertex is connected to a black vertex by an edge if their defining cycles have non-empty intersection. Then the $n$ edges in this bipartite graph can be thought of as darts, and the graph will turn out to be the dual graph to the topological hypermap defined below.

Indeed, a hypermap naturally appears in the setting of an orientable genus $g$ surface $\Sigma_g$ and a graph $\Gamma$ embedded in $\Sigma_g$ as $\iota: \Gamma \rightarrow \Sigma$, satisfying
\begin{itemize}
\item[1)] the complement $\Sigma_g \setminus \iota(\Gamma)$ is a union of topological discs called faces,
\item[2)] the faces are properly two-colourable (e.g. into black and white), i.e. faces of the same colour intersect only at vertices of $\Gamma$, and
\item[3)] the corners of the white faces are labelled with the numbers $1, 2, 3, \dots$ in some fashion (we may think that this information is carried by the embedding map $\iota$), and a black face corner label is equal to the adjacent white face corner label, when moving clockwise around their common vertex.
\end{itemize}
Then the triple $H = \langle \Sigma_g; \Gamma, \iota \rangle$ is an oriented labelled \textit{topological} hypermap. 

By removing a sub-disc in the interior of each face of a hypermap we obtain a \textit{ribbon graph} or \textit{``fat'' graph}, which is a graph together with a cyclic ordering on the set of darts incident to each vertex; the edges of the ribbon graph can be seen as small rectangles or ribbons attached in a given cyclic order to discs glued at the vertices. 

The correspondence between the topological and combinatorial definitions follows: 
\begin{itemize}
\item[1)] each disjoint cycle of $\alpha$ is obtained from recording the corner labels of a black face in an anticlockwise direction,
\item[2)] each disjoint cycle of $\sigma$ is obtained from recording the labels around a vertex in an anticlockwise direction,
\item[3)] each disjoint cycle of $\varphi$ is obtained from recording the corner labels of a white face in an anticlockwise direction. 
\end{itemize}

Consequently, the set of face labels becomes the set of darts of $H$, the white faces become hyper-faces of $H$ and the black faces become hyper-edges of $H$. Thus the combinatorial and topological descriptions of $H$ agree. Indeed, each topological hypermap produces a unique combinatorial hypermap, and given the above combinatorial information one may assemble an oriented connected surface from a number of topological discs, which are represented by polygons with labelled corners.

A topological hypermap $H$ is \textit{rooted} if we mark the first edge encountered while moving clockwise around the white corner of $H$ labelled $1$, and a combinatorial hypermap is rooted if one of its darts is marked as a root. We always assume that the root dart is $1$. 

\begin{Ex}
A partial picture of a rooted hypermap is shown in Figure~\ref{fig:hypermap}. The root corner is marked with an arrow and numbered $1$. 
\end{Ex}

\begin{figure}[h]
\centering
\includegraphics[scale=1.5]{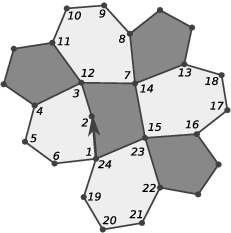}
\caption{A partial drawing of a rooted hypermap (its root is labelled $1$)}\label{fig:hypermap}
\end{figure}

 If we ignore the labelling (or marked root) of a hypermap, then we consider its entire isomorphism class. Two (combinatorial) hypermaps $H_1 = \langle D_1; \alpha_1, \sigma_1 \rangle$ and $H_2 = \langle D_2; \alpha_2, \sigma_2 \rangle$, assumed to be neither labelled nor rooted, are \textit{isomorphic} if there exists $\psi \in \mathfrak{S}_n$ such that $\psi\, \alpha_1 \psi^{-1}= \alpha_2\, $ and $\psi\, \sigma_1 \psi^{-1}= \sigma_2\, $. Their topological counterparts are isomorphic if there exists an orientation-preserving homeomorphism between the respective surfaces which preserves the graph embedding.
Two (combinatorial) rooted hypermaps are isomorphic if their roots correspond to each other under some hypermap isomorphism. An analogous definition holds in the topological case. Below we shall use the combinatorial and topological descriptions of hypermaps interchangeably.

A $(p,q)$-hypermap $H$ on $n$ darts is one in which $\alpha$ and $\varphi$ are permutations in $\mathfrak{S}_n$ such that $\alpha$ has $n/p$ cycles of length $p$, and $\varphi$ has $n/q$ cycles of length $q$. In other words, all hyper-edges of $H$ are $p$-gons and all hyper-faces of $H$ are $q$-gons.  Given this definition, it is more convenient to represent $H$ as $H = \langle D; \alpha, \varphi \rangle$.

\begin{figure}[h]
\centering
\includegraphics[scale=0.45]{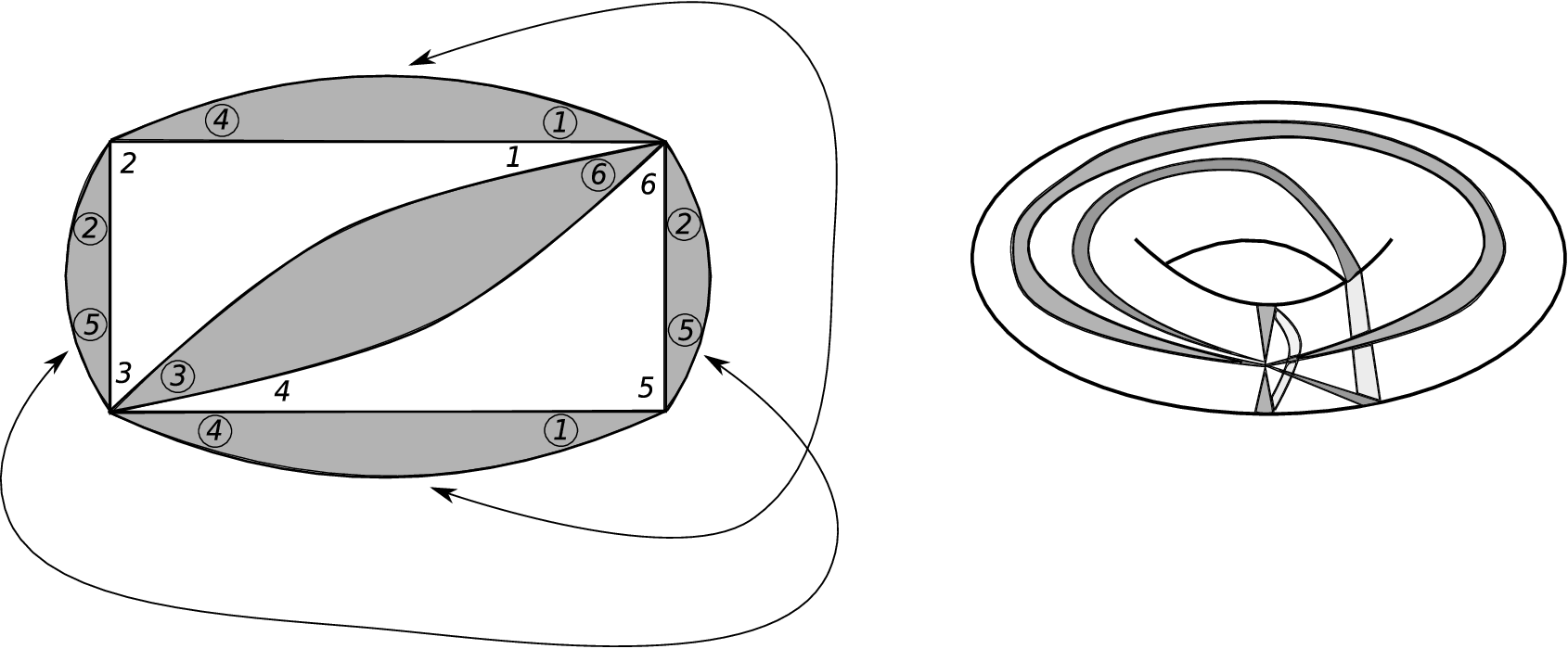}
\caption{A $(2,3)$-hypermap on a torus: the parallel sides of the rectangular shape on the left are identified in accordance with the circled labels and arrows in order to produce the torus with a hypermap on it depicted on the right.}\label{fig:triangulation}
\end{figure}

\begin{Ex}
The partial picture in Figure~\ref{fig:hypermap} features a $(5,6)$-hypermap.  By analogy with fullerene polyhedra, we shall call $(5,6)$-hypermaps \textit{fullerene hypermaps}, for short. However, in this case fullerene hypermaps do not have any specific restrictions (apart from those coming from the number of darts) on their topological genus, and thus on the number of pentagonal faces. 
\end{Ex}

\begin{Ex}
A triangulated surface carries a $(2,3)$-hypermap all of whose bigonal hyper-edges are collapsed into ordinary edges. We shall refer to a $(2,3)$-hypermap as a \textit{triangulation} (of an oriented surface), thus allowing identification of two sides of the same triangle. Figure~\ref{fig:triangulation} shows a triangulation of a torus with $\alpha = (1,4) (2,5) (3,6)$, $\sigma = (1,6,2,4,3,5)$, $\varphi = (1,2,3) (4,5,6)$.

Another important class of hypermaps is the class of $(2,4)$-hypermaps, or quadrangulations. In general, every $(2,q)$-hypermap is equivalent to a \textit{map} as described in \cite[\S 2 - \S 3]{Jones-Singerman}, and (once labelled) its corner labels of hyper-edges become exactly the dart labels of the resulting map after all bigonal hyper-edges are collapsed into usual edges.
\end{Ex}

\begin{figure}[h]
\centering
\includegraphics[scale=0.45]{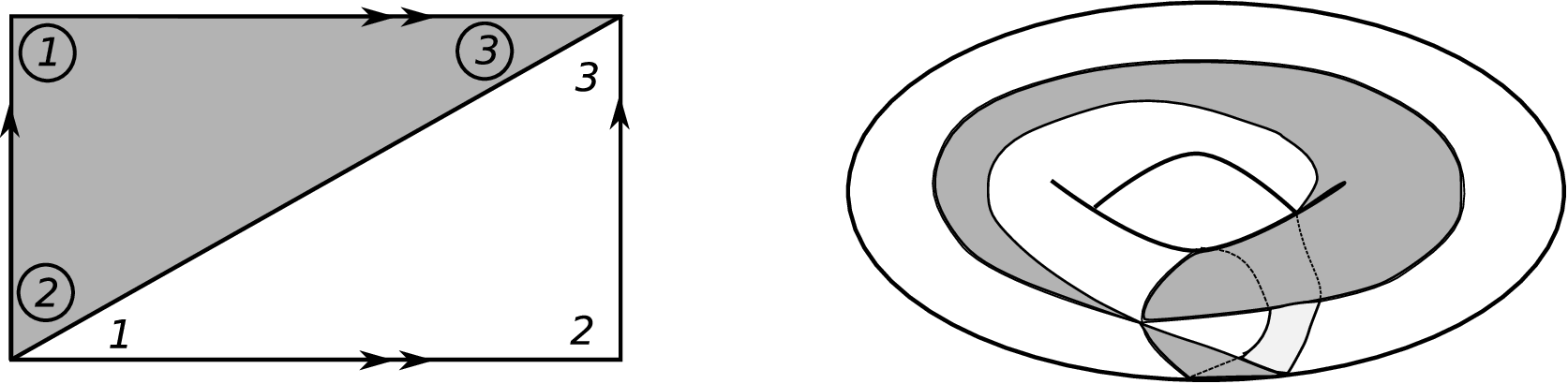}
\caption{A $(3,3)$-hypermap on a torus: the parallel sides of the rectangle on the left are identified according to the arrows marking them in order to produce the torus with a hypermap on it depicted on the right.}\label{fig:bicoloured-triangulation}
\end{figure}

\begin{Ex}
The $(3,3)$-hypemaps are triangulations that admit a colouring which is chequerboard around the vertices. We shall call $(3,3)$-hypermaps \textit{bi-coloured triangulations} (whose dual map is a bipartite cubic graph). Figure~\ref{fig:bicoloured-triangulation} features a bicoloured triangulation of a torus with $\alpha = \sigma = \varphi = (1,2,3)$. 
\end{Ex}

\subsection{Formal series}

A \textit{hypergeometric sequence} $(c_k)_{k\geq 0}$ is one for which $c_0=1$ and the ratio of consecutive terms is a rational function in $k$, i.e. there exist monic polynomials $P(k)$ and $Q(k)$ such that $$\frac{c_{k+1}}{c_k}=\frac{P(k)}{Q(k)}.$$
If $P$ and $Q$ can be factored as $$\frac{P(k)}{Q(k)}=\frac{(k+a_1)(k+a_2)\dots(k+a_p)}{(k+b_1)(k+b_2)\dots(k+b_q)(k+1)},$$
with $a_i$, $1 \leq i \leq p$, $b_j$, $1 \leq j \leq q$, non-negative real numbers, then we use the notation 
$$ {}_{p}F_{q}\left[ \begin{array}{c} a_1 \dots a_p\\ b_1 \dots b_q 
\end{array}; z \right]$$ for the formal series $F(z)=\sum_{k\geq 0} c_kz^k$, c.f. \cite[\S 3.2]{Zeilberger-et-al}. The factor $(k+1)$ belongs to the denominator for historical reasons. Such a hypergeometric series satisfies the differential equation
\begin{equation}\label{HGSdiffeq}
\Big(\vartheta(\vartheta+b_{1}-1)\cdots(\vartheta+b_{q}-1)-z(\vartheta+a_{1})%
\cdots(\vartheta+a_{p})\Big)\,\, {}_{p}F_{q}(z) = 0,
\end{equation}
where $\vartheta=z\frac{d}{dz}$, c.f. \cite[\S 16.8(ii)]{DLMF}. Among numerous differential equations related to (\ref{HGSdiffeq}) is the \textit{classical Riccati equation}, which will play an important role in this paper. It is a first order non-linear equation with variable coefficients $f_i(x)$, of the form 
\begin{equation}\label{Riccati2}
\frac{\mathrm{d}y}{\mathrm{d}x}= f_1(x) + f_2(x) y + f_3(x) y^2.
\end{equation}

The \textit{Pochhammer symbol} is related to hypergeometric series, is defined as $$(a)_n=a(a+1)\dots (a+n-1),$$ and has asymptotic expansion
\begin{equation}\label{Pochhammer}
(a)_n \approx \frac{\sqrt{2\pi}}{\Gamma(a)}\, e^{-n}\, n^{a + n - \frac{1}{2}},
\end{equation}
where $\Gamma(a)$ is the Gamma function of $a$, defined as $\Gamma(a)=(a-1)!$ if $a$ is a positive integer, and $\Gamma(a)= \int_0^{\infty} x^{a-1}e^{-x} dx$ for all the non-integer real positive numbers.

A formal power series $y=f(x)$ is said to be \textit{D-finite}, or \textit{differentiably finite}, or \textit{holonomic}, if there exist polynomials $p_0, \dots, p_m$ (not all zero) such that $p_m(x)y^{(m)}+ \dots +p_0(x)y=0$, where $y^{(m)}$ denotes the $m$-th derivative of $y$ with respect to $x$. All algebraic power series are holonomic, but not vice versa, c.f. \cite[Appendix B.4]{Flajolet-Sedgewick}.

Finally recall that the \textit{Hadamard product} $(A \odot B)(z)$ of two formal single-variable series $A(z)=\sum_{n\geq 0} a_n \frac{z^n}{n!}$ and $B(z)=\sum_{n\geq 0} b_n \frac{z^n}{n!}$ is given by $$(A \odot B)(z):=\sum_{n\geq 0} a_n b_n \frac{z^n}{n!}.$$

Let $\lambda = 1^{n_1} 2^{n_2} \dots m^{n_m}$ be a partition of a natural number $n\geq 0$, i.e. $n = \sum_{i\geq 1}\, i n_i$. We write $\lambda \vdash n$ and define $\lambda ! := 1^{n_1} n_1! 2^{n_2} n_2! \dots m^{n_m} n_m!$. Let $\mathbf{z}^{\lambda}: = z_1^{n_1} z_2^{n_2} \dots z_m^{n_m}$ for some collection of variables $z_1$, $z_2$, $\dots$. Then for two multi-variable series $A(\mathbf{z}) = \sum_{n\geq 0} \sum_{\lambda \vdash n} a_\lambda \frac{\mathbf{z}^\lambda}{\lambda !}$ and $B(\mathbf{z}) = \sum_{n\geq 0} \sum_{\lambda \vdash n} b_\lambda \frac{\mathbf{z}^\lambda}{\lambda !}$ we have $$(A\odot B)(\mathbf{z}) := \sum_{n\geq 0} \sum_{\lambda \vdash n} a_\lambda b_\lambda \frac{\mathbf{z}^\lambda}{\lambda!}.$$

\subsection{Species theory}

Species theory (th\'{e}orie des esp\`{e}ces) is due to A. Joyal \cite{Joyal} and is a powerful way to describe and count labelled discrete structures. Since it requires a lengthy and formal setup, we give here only the basic ideas and refer the reader to \cite{Bergeron-et-al, Flajolet-Sedgewick} for further details.

A \textit{species of structures} is a rule (or functor) $F$ which produces

\begin{itemize}
\item[i)] for each finite set $U$ (of labels), a finite set $F[U]$ of structures on $U$,
\item[ii)] for each bijection $\sigma: U\rightarrow V$, a function $F[\sigma]: F[U] \rightarrow F[V]$.
\end{itemize} 
 
The functions $F[\sigma]$ should further satisfy the following functorial properties:

\begin{itemize}
\item[i)] for all bijections $\sigma:U \rightarrow V$ and $\tau:V \rightarrow W$, 
$F[\tau \circ \sigma] = F[\tau]\circ F[\sigma]$,
\item[ii)] for the identity map $Id_U : U \rightarrow U$, $F[Id_U] = Id_{F[U]}$.
\end{itemize}

Let $[n] = \{1,2,\dots,n\}$ be an $n$-element set, and assume that $[0] = \emptyset$. A species $F$ of \textit{labelled structures} has exponential generating function $F(z) = \sum_{n \geq 0} |F[n]| \frac{z^n}{n!}$. For a species of \textit{unlabelled structures} (i.e. structures up to isomorphism) we write $\widetilde{F}$, and its generating function is a specialisation of the cycle index series, in the sense that $\widetilde{F}(z)=\mathcal{Z}_F(z,z^2, z^3 \dots)$, where the \textit{cycle index series} (see \cite[\S 1.2.3]{Bergeron-et-al}) is defined as: $$ \mathcal{Z}_F(z_1,z_2, \dots) = \sum_{n\geq 0} \frac{1}{n!} \sum_{\sigma \in \mathfrak{S}_n} | Fix(F[\sigma])| \mathbf{z}^\sigma.$$ Here $Fix(F[\sigma])$ is the set of elements of $F[n]$ having $F[\sigma]$ as automorphism, and $\mathbf{z}^\sigma = z_1^{c_1} z_2^{c_2} \dots z_m^{c_m}$ if the cycle type of $\sigma$ is $c(\sigma) = (c_1, c_2, \dots, c_m)$ (i.e. $c_k$ is the number of cycles of length $k$ in the decomposition of $\sigma$ into disjoint cycles). 

Species can often be described by functional equations, as in the following example. 

\begin{Ex}
Let $\mathcal{A}$ denote the species of rooted trees (i.e. trees with a distinguished vertex, or \textit{arborescences} \cite{Joyal}), and $E$ the species of sets (from the French \textit{ensembles} \cite{Joyal}). Let $Z$ be the singleton species with generating function $Z(z) = z$. Then the functional equation $\mathcal{A} = Z E(\mathcal{A})$ expresses the fact that any rooted tree with vertex labels from a finite set $U$ can be naturally described as a root (a vertex $z \in U$) to which is attached a set of disjoint rooted trees (on $U\setminus \{z\}$) which translate into equalities for generating functions; in this case we have $\mathcal{A}(z) = z \exp(\mathcal{A}(z))$, where $\mathcal{A}(z)$ is the generating function for finite rooted labelled trees.

By using the Lagrange -- Br\"{u}nner inversion formula we get $\mathcal{A}(z) = \sum_{n\geq 2} \frac{n^{n-2}}{(n-1)!} z^n$. This leads to Cayley's formula of $n^{n-2}$ for the number of labelled trees on $n$ vertices via the fact that the number of rooted trees on $n$ vertices is the $n$-th coefficient of $\mathcal{A}(z)$ and each tree with $n$ vertices rooted at $1$ corresponds to $(n-1)!$ labelled trees.
\end{Ex}

\section{Subgroups of free products of cyclic groups}\label{s:groups}

Let $\mathfrak{H}^r_{p,q}(n)$ be the set of connected oriented rooted $(p,q)$-hypermaps on $n$ darts. In this section we prove two lemmas about the correspondence between finite index subgroups of $\Delta^+$ and hypermaps. This correspondence may also be seen via arguments involving the Schreier graph of the respective subgroup (c.f. \cite{Stothers}), however additional steps are required in order to translate the Schreier graph into the corresponding hypermap.

\begin{Lem}\label{lemma:hypermaps-and-subgroups-1}
Let $p, q > 1$ be two natural numbers such that $pq \geq 6$. There is a one-to-one correspondence between $\mathfrak{H}^r_{p,q}(n)$ and the set of free subgroups of index $n$ in the group $\Delta^+ = \mathbb{Z}_p * \mathbb{Z}_q$.
\end{Lem}
\begin{proof}
Let $H = \langle D; \alpha, \varphi \rangle$ be a rooted hypermap (with root $1$) from $\mathfrak{H}^r_{p,q}(n)$. Then there is an epimorphism from $\Delta^+ = \mathbb{Z}_p * \mathbb{Z}_q \cong \langle \varrho | \varrho^p = \varepsilon \rangle * \langle \delta | \delta^q = \varepsilon \rangle$ to the group $G(H) = \langle \alpha, \varphi \rangle$ given by $\varrho \mapsto \alpha$, $\delta \mapsto \varphi$, and $\Delta^+$ acts transitively on $D$ via this epimorphism. Let $\Gamma$ be the stabiliser subgroup of dart $1$. Then $[\Delta^+ : \Gamma]=n$ and the action of $\Delta^+$ on $D$ is equivalent to the action of $\Delta^+$ on the cosets of $\Gamma$. 

Moreover, $\Gamma$ cannot contain any conjugate of a non-trivial power of $\varrho$ or $\delta$ because $\alpha$ and $\varphi$ have no fixed points, and cycle structure is preserved under conjugation. Thus $\Gamma$ has no torsion elements, as any torsion element in $\Delta^+$ is conjugate to some power of either $\varrho$ or $\delta$. By the Kurosh theorem on subgroups of free products, $\Gamma$ is free.

On the other hand, a torsion-free finite index subgroup $\Gamma < \Delta^+$ gives rise to a combinatorial hypermap $H = \langle D_\Gamma; \alpha_\Gamma, \varphi_\Gamma \rangle$, with $D_\Gamma = \{ g\,\Gamma | g \in \Delta^+ \}$, $\alpha_\Gamma(g\Gamma) = (\varrho g) \Gamma$, $\varphi_\Gamma(g\Gamma) = (\delta g) \Gamma$. The root of $H$ corresponds to the coset $\varepsilon \Gamma$. 

Since $\Gamma$ is torsion-free, it does not contain any conjugates of $\varrho$, $\delta$, or their powers. Thus, $\alpha^p_\Gamma = \varepsilon$ and $\alpha^k_\Gamma \neq \varepsilon$ for $1 \leq k < p$, which implies that all disjoint cycles of $\alpha_\Gamma$ have length $p$. Indeed, a cycle in $\alpha_\Gamma$ has length $d$, $d \nmid p$, and once $d < p$, then $\alpha^d_\Gamma$ has fixed points. Thus $\Gamma$ contains a conjugate of $\varrho^d$, which is a contradiction to $\Gamma$ being torsion-free. Analogously, all disjoint cycles of $\varphi_\Gamma$ have length $q$, so it follows that $H \in \mathfrak{H}^r_{p,q}(n)$, where $n = [\Delta^+:\Gamma]$; then we use again that a torsion-free subgroup $\Gamma < \Delta^+$ is free.  
\end{proof}

By the definition of hypermap isomorphism, we analogously obtain the following. 

\begin{Lem}\label{lemma:hypermaps-and-subgroups-2}
Let $p, q > 1$ be two natural numbers such that $pq \geq 6$. There is a one-to-one correspondence between the set of isomorphism classes of connected oriented $(p,q)$-hypermaps on $n$ darts and the set of conjugacy classes of free subgroups of index $n$ in the group $\Delta^+ = \mathbb{Z}_p * \mathbb{Z}_q$.
\end{Lem}

\section{Counting free subgroups and hypermaps}

We proceed by computing the numbers $|\mathfrak{H}^r_{p,q}(n)|$ of connected oriented rooted $(p,q)$-hypermaps on $n$ darts. In order to do so we shall use species theory (see Section 2.3) and generalise the results of \cite{Petitot-Vidal}.

Let $E$ be the species of sets, $C_i$ the species of cyclic permutations of length $i\geq 2$, $S_i$ the species of permutations with cycles of length $i$ only and no fixed points (assume $S_i[\emptyset] = \{ \emptyset \}$), $H^{*}$ the species of labelled $(p,q)$-hypermaps (not necessarily connected) on $n$ darts (assume $H^*[\emptyset] = \{\emptyset\}$), $H$ the species of connected labelled $(p,q)$-hypermaps on $n$ darts (in contrast, here $H[\emptyset] = \emptyset$), and $H^{\circ}$ the species of connected rooted $(p,q)$-hypermaps on $n$ darts. Thus, an instance of species $H^{*}$ is a (possibly empty) disjoint union of several instances of species $H$, while an instance of $H^{\circ}$ is an instance of $H$ with all but one of its labels discarded (that is, only the root has a label). 

The following combinatorial equations describe the relations between the species:
\begin{equation}\label{eq:species-1}
S_p = E(C_p),\,\, S_q = E(C_q),\,\, H^* = S_p \times S_q.
\end{equation}
Intuitively, this means that each permutation of $S_p$ has a unique decomposition into cycles of length $p$, and each hypermap is uniquely determined by a pair of permutations from $S_p$ and $S_q$.  
Furthermore
\begin{equation}\label{eq:species-2}
H^* = E(H),\,\, H^\circ = Z\cdot H^\prime,
\end{equation}
where $Z$ is the singleton species with generating function $Z(z) = z$, and $H^\prime$ means species differentiation, c.f. \cite[\S 1.4, D\'{e}finition 5]{Bergeron-et-al}.

Since the generating function for the species of sets $E$ is $\exp(z)$ \cite[\S 1.2, Exemples 2]{Bergeron-et-al}, 
the respective generating functions will be 
\begin{equation}\label{eq:gen-func-1}
C_p(z) = \frac{z^p}{p},\,\, C_q(z) = \frac{z^q}{q},
\end{equation}
and thus
\begin{equation}\label{eq:gen-func-2}
S_p(z) = \exp\left( \frac{z^p}{p} \right) = \sum^{\infty}_{k=0} \frac{z^{pk}}{p^k}\,\frac{1}{k!},
\end{equation}
\begin{equation}\label{eq:gen-func-3}
S_q(z) = \exp\left( \frac{z^q}{q} \right) = \sum^{\infty}_{k=0} \frac{z^{qk}}{q^k}\,\frac{1}{k!}.
\end{equation}
 
We shall use the notation
$$ \langle p,q \rangle:=lcm(p,q) \ \textrm{and} \  (p,q):= gcd(p,q),$$
where \textit{lcm} and \textit{gcd} denote as usual the least common multiple and greatest common divisor, respectively.

Identities (\ref{eq:species-1}), (\ref{eq:gen-func-2}) and (\ref{eq:gen-func-3}) imply that
\begin{equation}\label{eq:gen-func-4}
H^*(z) = S_p(z)\odot S_q(z) = \sum^{\infty}_{k=0} z^{\langle p,q \rangle k}\,\, \frac{(\langle p,q \rangle k)!}{p^{k \langle p,q \rangle/p} \left( \frac{\langle p,q \rangle k}{p} \right)!\,\, q^{k \langle p,q \rangle/q} \left( \frac{\langle p,q \rangle k}{q} \right)!},
\end{equation}
where $\odot$ denotes the Hadamard product of $S_p(z)$ and $S_q(z)$, and is the functional version of a Cartesian product.
Note that we can express $H^*(z)$ as $ H^*(z)= f(c z^{\langle p,q \rangle})$ for a suitable constant $c > 0$, and $f(z) = \sum^{\infty}_{k=0} f_k \frac{z^k}{k!}$. If we choose the constant $c$ so that 
\begin{equation}\label{eq:gen-func-5}
c^{( p,q )} = \langle p,q \rangle^{(p-1)(q-1)-1},
\end{equation}
then for $k \geq 0$
\begin{equation}\label{eq:gen-func-6}
\frac{f_{k+1}}{f_k} = \frac{ \prod^{\frac{pq}{(p,q)}-1}_{i=1} \, \left( k + \frac{i (p,q)}{pq} \right) }{ \prod^{\frac{p}{(p,q)}-1}_{i=1}\, \left( k + \frac{i (p,q)}{p} \right) \,\, \prod^{\frac{q}{(p,q)}-1}_{i=1}\, \left( k + \frac{i (p,q)}{q} \right) }.
\end{equation}
Since $f(0)=1$, the series $f(z)$ is hypergeometric, and we obtain
\begin{equation}\label{eq:gen-func-7}
f(z) = {}_{\frac{pq}{(p,q)} - 1}F_{\frac{p+q}{(p,q)} - 2}\left[ \begin{array}{cc}
\frac{i(p,q)}{pq},& i=1,\dots, \frac{pq}{(p,q)}-1\\
\frac{j(p,q)}{p}, j=1,\dots, \frac{p}{(p,q)}-1;& \frac{j(p,q)}{q}, j=1,\dots, \frac{q}{(p,q)}-1
\end{array}; z \right].
\end{equation} 

Let us observe that whenever $p$ or $q$ is a divisor of $i$, \eqref{eq:gen-func-6} can be simplified so that $\floor*{ \frac{q}{(p,q)} - \frac{1}{p} } + \floor*{ \frac{p}{(p,q)} - \frac{1}{q} } = \frac{p+q}{(p,q)} - 2$ terms from both its numerator and denominator cancel out. Thus \eqref{eq:gen-func-7} turns into
\begin{equation}\label{eq:gen-func-8}
f(z) = {}_{N}F_0\left[ \begin{array}{c}
\frac{i(p,q)}{pq}, i=1,\dots,\frac{pq}{(p,q)}-1, p \nmid i, q \nmid i\\
\ldots
\end{array}; z \right],
\end{equation}
with $N = 1 + \frac{(p-1)(q-1) - 1}{(p,q)}$.

Since $f(z)$ is hypergeometric and thus holonomic, the growth function $H^*(z) = f(x)$, with $x = c\, z^{\langle p,q \rangle}$, is also holonomic. It is also clear that $f(x)$ has convergence radius $0$. Thus, by \cite[Corollary 2]{Harris-Sibuya}, $\frac{1}{f(x)}$ is non-holonomic, and $\frac{f^\prime(x)}{f(x)}$ is transcendental (c.f. also \cite[Proposition 3.1.5]{Ch-thesis}). 

From \eqref{eq:species-2} we have the first equality below, and from $H^*(z) = f(x)$ the second:
\begin{equation}\label{eq:gen-func-9}
H^\circ(z) = z\, \frac{d}{dz} \log H^*(z)=\langle p,q \rangle\, x\, \frac{f^\prime(x)}{f(x)}.
\end{equation}
Thus $H^\circ(z)$ is transcendental, and we have shown the following.

\begin{Thm}\label{thm:p-q-hypermaps-growth-transcendental}
The generating series for the number of connected rooted $(p,q)$-hypermaps on $n$ darts $$H^\circ(z) = \sum^{\infty}_{n=0} | \mathfrak{H}^r_{p,q}(n)| z^n$$ is transcendental, and the asymptotic expansion for its non-zero coefficients is 
\begin{equation*}
[z^{\langle p,q \rangle k}] H^\circ(z) \approx c_0 k^{c_1 k + \frac{1}{2}}\,\, e^{ c_2 k}, \mbox{ as } k \rightarrow \infty,
\end{equation*}
where $c_0, c_1, c_2 > 0$ are constants that depend only on $p$ and $q$.
\end{Thm}

\begin{proof}
The proof of the formula and transcendence is given above; it only remains to justify the asymptotics of $[z^{\langle p,q \rangle k}] H^\circ(z)$. The series $H^\circ(z)$ has non-negative coefficients and is rapidly divergent, and thus we may apply the technique described in \cite[Theorem 4.1]{Drmota-Nedela} (see also \cite{Bender} and \cite[Theorem~7.2]{Odlyzhko}). Then the coefficients of $H^*(z)$ are given by 
\begin{equation}\label{eq:asympt-1}
[z^{\langle p,q \rangle k}] H^*(z) = \frac{c^k}{k!} \, \prod^N_{i=1} (a_i)_k \, 
\end{equation}  
and using the asymptotic expansion for the Pochhammer symbol in (\ref{Pochhammer}) we obtain the asymptotic expansion for the coefficients of $\log H^*(z)$, from \eqref{eq:asympt-1} and subsequently for the coefficients of $H^{\circ}(z) = z\, \frac{d}{dz} \log H^*(z)$ as 
\begin{equation}\label{eq:asympt-2}
[z^{\langle p, q \rangle k}] H^\circ(z) \approx \frac{ (2\pi)^{\frac{N-1}{2}} \langle p,q \rangle }{\prod_i \Gamma(a_i)} \,\, k^{(N-1)k - \frac{N-1}{2} + \sum_i a_i}\,\, e^{ (1 + \log c - N) k}, \mbox{ as } k \rightarrow \infty,
\end{equation}
where $N = 1 + \frac{(p-1)(q-1) - 1}{( p, q )}$, and the constants $c$ and $a_i$ satisfy $c^{( p, q )} = \langle p, q \rangle^{(p-1)(q-1)-1}$, $a_i = \frac{i ( p, q )}{pq}$, with $i = 1,\dots, \frac{pq}{( p, q )}-1$, $p \nmid i$, $q \nmid i$. Also, notice that $\sum_i a_i = \frac{N}{2}$, and thus the above formula takes the required form.
\end{proof}

By Lemma~\ref{lemma:hypermaps-and-subgroups-1}, since the elements of $\mathfrak{H}^r_{p,q}(n)$ are in bijection with free subgroups of index $n$ in $\Delta^+$, we obtain:

\begin{Thm}\label{thm:subgroup-growth-transcendental}
Let $\Delta^+ = \mathbb{Z}_p * \mathbb{Z}_q$ be a free product of cyclic groups with $p, q > 1$, $pq\geq 6$. Then the subgroup growth series $S_{f}(z) = \sum^{\infty}_{n=0} s_{f}(n) z^n$ for the numbers $s_{f}(n)$ of free subgroups of index $n$ of $\Delta^+$ coincides with $H^\circ(z)$.
\end{Thm}
Below, we consider several examples of the generating series $H^\circ(z) = \sum^{\infty}_{n=0} | \mathfrak{H}^r_{p,q}(n)| z^n$ given by formulas \eqref{eq:gen-func-8} -- \eqref{eq:gen-func-9} for various values of $p$ and $q$. In order to facilitate computations with power series, we use the SAGE code from Appendix~I. 

\begin{Ex}\label{example:rooted-triangulations}
For the number of rooted triangulations on $n$ darts we have $H^\circ(z) = \sum^{\infty}_{n=0} |\mathfrak{H}^r_{2,3}(n)|z^n = 5\, z^6 + 60\, z^{12} + 1105\, z^{18} + 27120\, z^{24} + 828250\, z^{30} + 30220800\, z^{36} + 1282031525\, z^{42} + 61999046400\, z^{48} + 3366961243750\, z^{54} + 202903221120000\, z^{60} + \dots$. The coefficient sequence of $H^\circ(z)$ has index A062980 in the OEIS \cite{OEIS}.

Note that, in this case, if we explicitly compute all the constants in the asymptotic equality \eqref{eq:asympt-2} from the proof of Theorem \ref{thm:p-q-hypermaps-growth-transcendental}, then $| \mathfrak{H}^r_{2,3}(6k) |\approx C \frac{6^k}{e^k}k^{k-\frac{1}{2}}$. The latter coincides with the formula for free subgroups of index $6k$ in the modular group from \cite[p.~1277]{Stothers}, after applying Stirling's approximation.
\end{Ex}

\begin{Ex}\label{example:rooted-quadrangulations}
For the number of rooted quadrangulations on $n$ darts we have $H^\circ(z) = \sum^{\infty}_{n=0} |\mathfrak{H}^r_{2,4}(n)| z^n = 3\, z^4 + 24\, z^8 + 297\, z^{12} + 4896\, z^{16} + 100278\, z^{20} + 2450304\, z^{24} + 69533397\, z^{28} + 2247492096\, z^{32} + 81528066378\, z^{36} + 3280382613504\, z^{40} + \dots$. The coefficient sequence of $H^\circ(z)$ has index A292186 in the OEIS \cite{OEIS}.
\end{Ex}

\begin{Ex}\label{example:rooted-bicoloured-triangulations}
For the number of rooted bi-coloured triangulations on $n$ darts we have $H^\circ(z) = \sum^{\infty}_{n=0} |\mathfrak{H}^r_{3,3}(n)| z^n = 2\, z^3 + 12\, z^6 + 112\, z^9 + 1392\, z^{12} + 21472\, z^{15} + 394752\, z^{18} + 8421632\, z^{21} + 204525312\, z^{24} + 5572091392\, z^{27} + 168331164672\, z^{30} + \dots$. The coefficient sequence of $H^\circ(z)$ has index A292187 in the OEIS \cite{OEIS}.
\end{Ex}

\begin{Ex}\label{example:rooted-fullerene-hypermaps}
For the number of rooted fullerene hypermaps on $n$ darts we have $H^\circ(z) = \sum^{\infty}_{n=0} |\mathfrak{H}^r_{5,6}(n)| z^n = 
758038579710193926144\, z^{30} + 194568955255295107105$ $14063909668515373421277741056000\, z^{60} + 8904485565951809034397866445591191272830$ $319687215549893902610979100158748795888625254400\, z^{90} + \dots$.
\end{Ex}

\begin{Ex}\label{example:subgroup-growth-series}
The growth series counting the finite index free subgroups in $\mathbb{Z}_2*\mathbb{Z}_3$, $\mathbb{Z}_2*\mathbb{Z}_4$, $\mathbb{Z}_3*\mathbb{Z}_3$ and $\mathbb{Z}_5*\mathbb{Z}_6$ are given in Examples \ref{example:rooted-triangulations}, \ref{example:rooted-quadrangulations}, \ref{example:rooted-bicoloured-triangulations} and \ref{example:rooted-fullerene-hypermaps}, respectively.   
\end{Ex}

\section{The Riccati hierarchy and recurrence relations}\label{s:recurrence}

As we mentioned before, the generating series $H^\circ(z)$ cannot be algebraic. Nevertheless, it does satisfy a non-linear \textit{differential} equation, and although this fact does not make $H^\circ(z)$ holonomic, a certain recurrence relation holds for its coefficients $h_{p,q}(n) = | \mathfrak{H}^r_{p,q}(n)|$.

Recall that we expressed $H^*(z)$ as  $H^*(z)= f(x)$, where $x = c\, z^{\langle p,q \rangle}$. By \eqref{eq:gen-func-5} -- \eqref{eq:gen-func-8},  the function $f(x)$ is hypergeometric and by (\ref{HGSdiffeq}) it satisfies the hypergeometric differential equation 
\begin{equation}\label{eqn:riccati-1}
\vartheta f(x) = x\,\, \prod^N_{i=1} (\vartheta + a_i) f(x),
\end{equation} 
where $\vartheta = x\, \frac{d}{dx}$, and $a_i$, $i=1, \dots, N$, are the parameters of the hypergeometric function in equality \eqref{eq:gen-func-8}, with $N = 1 + \frac{(p-1)(q-1)-1}{( p,q )}$. 

Let $w(x):= x \frac{f'(x)}{f(x)}$. Then $w(x)$ determines the growth series $H^\circ(z) = \sum_{n\geq 0} h_{p,q}(n)\, z^n$ since \eqref{eq:gen-func-9} leads immediately to
\begin{equation}\label{eqn:riccati-2}
H^\circ(z) = \langle p,q \rangle\, x\, \frac{f'(x)}{f(x)} = \langle p,q \rangle\ w(x).
\end{equation}

\begin{Lem}
The function $w(x)$ satisfies 
\begin{equation}\label{eqn:riccati-3}
w(x) = x\,\, \sum^N_{i=0} \sigma_{N-i}(a_1, \dots, a_N)\, w_i(x), 
\end{equation} 
where $\sigma_i$ is the $i$-th elementary symmetric function, and the functions $w_i(x)$ are defined as
\begin{equation}\label{eqn:riccati-4}
w_0(x) = 1,\,\, w_1(x) = w(x),\,\, w_i(x) = x \,\, \frac{d}{dx} w_{i-1}(x) + w(x) w_{i-1}(x), \mbox{ for } i\geq 2.
\end{equation}
\end{Lem}

\begin{proof}
By (\ref{eqn:riccati-1}) 
\begin{equation*}
\vartheta f(x) = x\, \prod^N_{i=1} (\vartheta + a_i) f(x) = x\, \sum^N_{i=0} \sigma_{N-i}(a_1, \dots, a_N)\, \vartheta^i f(x),
\end{equation*}
and from the definition of $w(x)$ it follows that $\vartheta f(x) = w(x)\, f(x) = w_1(x)\, f(x)$; we will prove by induction that $\vartheta^i f(x)= w_{i}(x)\cdot f(x)$, assuming it holds for all values up to $i-1$:
\begin{equation*}
\vartheta^i f(x) = \vartheta( w_{i-1}(x)\cdot f(x) ) = \vartheta w_{i-1}(x)\cdot f(x) + w_{i-1}(x)\cdot \vartheta f(x) = 
\end{equation*}
\begin{equation*}
= \left( x\,\, \frac{d}{dx} w_{i-1}(x) + w(x) w_{i-1}(x) \right)\,\, f(x) = w_{i}(x)\cdot f(x).
\end{equation*}

Now \eqref{eqn:riccati-3} -- \eqref{eqn:riccati-4} follow immediately.
\end{proof}

The family of equations \eqref{eqn:riccati-3} -- \eqref{eqn:riccati-4} is indexed by the integers $p, q > 1$, $pq \geq 6$, and the first instance has $p=2$, $q=3$; this turns out to be the classical Riccati equation
\begin{equation}\label{eqn:riccati-5}
w(x) = x^2 w'(x) + x w(x) + x w^2(x) + \frac{5}{36}x.
\end{equation}

All the equations \eqref{eqn:riccati-3} -- \eqref{eqn:riccati-4} constitute a part of the Riccati hierarchy (c.f. \cite[Equation (7.2)]{GdL}), and for fixed $p,q$ they give rise to non-linear recurrence relations, with polynomials in $n$ as coefficients, for the numbers $h_{p,q}(n)$. Such equations (and, subsequently, the associated recurrence relations) can be computed with the help of the SAGE code in Appendix~II. 

We would like to remark that $w(x)$ is a solution to a \textit{single} appropriate equation in the Riccati hierarchy \eqref{eqn:riccati-3} -- \eqref{eqn:riccati-4}, and not \textit{each} of them, as one may expect by analogy to the KP hierarchy \cite{Goulden-Jackson, Okounkov}.

\begin{Ex}\label{example:recurrence-2-3}
If we consider the formal series expansion $w(x) = \sum^\infty_{n=0} w_n\, x^n$, then by the Riccati equation \eqref{eqn:riccati-5} we obtain
\begin{equation*}
w_{n+1} = (n+1)\, w_n + \sum^n_{i=0} w_i w_{n-i}, \mbox{ for } n\geq 2,
\end{equation*}
with initial conditions $w_0 = 0$ and $w_1 = \frac{5}{36}$ (c.f. \cite{Petitot-Vidal}). Thus
\begin{equation}\label{eqn:riccati-6}
h_{2,3}(6n+6) = 6(n+1)\, h_{2,3}(6n) + \sum^n_{i=0} h_{2,3}(6i)\, h_{2,3}(6n-6i), \mbox{ for } n\geq 2,
\end{equation} 
with $h_{2,3}(0) = 0$, $h_{2,3}(6) = 5$, and $h_{2,3}(d) = 0$ for any non-zero $d \neq 0 \mod 6$. 

This is a recurrence relation for the number of rooted triangulations on $n$ darts (equivalently, $n/3$ triangles). The corresponding number sequence can be also expressed as the $S(6, -8, 1)$ sequence from \cite{Martin}.
\end{Ex}

\begin{Ex}\label{example:recurrence-2-4}
In the case $p=2, q=4$ we arrive at the following equation for $w(x)$:
\begin{equation}\label{eqn:riccati-7}
w(x) = x^2 w'(x) + x w(x) + x w^2(x) + \frac{3}{16}\, x,
\end{equation}
which is also a classical Riccati equation. It gives rise to the following relation between the coefficients of the series $w(x) = \sum^\infty_{n=0} w_n\, x^n$:
\begin{equation*}
w_{n+1} = (n+1)\, w_n + \sum^n_{i=0} w_i w_{n-i}, \mbox{ for } n\geq 2,
\end{equation*}
with initial conditions $w_0 = 0$ and $w_1 = \frac{3}{16}$. Thus
\begin{equation}\label{eqn:riccati-8}
h_{2,4}(4n+4) = 4(n+1)\, h_{2,4}(4n) + \sum^n_{i=0} h_{2,4}(4i)\, h_{2,4}(4n-4i), \mbox{ for } n\geq 2,
\end{equation}
with $h_{2,4}(0) = 0$, $h_{2,4}(4) = 3$, and $h_{2,4}(d) = 0$ for any non-zero $d \neq 0 \mod 4$, 
and we obtained a recurrence for the number of rooted quadrangulations on $n$ darts (equivalently, $n/4$ squares). 
The corresponding number sequence can be also expressed as the $S(4, -6, 1)$ sequence from \cite{Martin}.
\end{Ex}

\begin{Ex}\label{example:recurrence-3-3}
In the case $p=q=3$ we arrive at yet another Riccati equation for $w(x)$:
\begin{equation}\label{eqn:riccati-9}
w(x) = x^2 w'(x) + x w(x) + x w^2(x) + \frac{2}{9}\, x,
\end{equation}
which gives rise to the recurrence relation for the number of rooted bi-coloured triangulations on $n$ darts (equivalently, with $n/3$ triangles):
\begin{equation}\label{eqn:riccati-10}
h_{3,3}(3n+3) = 3(n+1)\, h_{3,3}(3n) + \sum^n_{i=0} h_{3,3}(3i)\, h_{3,3}(3n-3i), \mbox{ for } n\geq 2,
\end{equation}
with $h_{3,3}(0) = 0$, $h_{3,3}(3) = 2$, and $h_{3,3}(d) = 0$ for any non-zero $d \neq 0 \mod 3$.
The corresponding sequence can be also expressed as the $S(3, -5, 1)$ sequence from \cite{Martin}.
\end{Ex}

Now consider the generating function $H^\circ(z)$ for any family of $(p,q)$-hypermaps associated with the classical Riccati equation, that is, the case when $N = 2$ in \eqref{eqn:riccati-3} -- \eqref{eqn:riccati-4}, or equivalently, when $p$ and $q$ satisfy the identity
\begin{equation*}
(p-1)\,\,(q-1) = ( p,q ) + 1.
\end{equation*}
An easy calculation shows that the only possible values are (i) $p=2, q=3$, (ii) $p=2, q=4$ and (iii) $p=3, q=3$, which correspond to Examples \ref{example:recurrence-2-3}, \ref{example:recurrence-2-4} and \ref{example:recurrence-3-3}, respectively. These are exactly the cases when our approach produces recurrence relations similar to \cite[Formula 9]{Stothers}. In each case the Riccati equation has the form 
\begin{equation}
w'(x)=\frac{w(x) - x w(x) - x w^2(x) - kx}{x^2},
  \end{equation}
  where $k$ is a constant. This is the form of the equation in \cite[Theorem 5.2]{Klazar} of Klazar, so one can easily deduce the following:
  \begin{Cor}\label{cor:non-holonomic}
The generating function $H^\circ(z)$ for any family of $(p,q)$-hypermaps associated with the classical Riccati equation is non-holonomic.
\end{Cor}
\begin{Que} Are the generating functions $H^\circ(z)$ associated with higher-order equations in the Riccati hierarchy also non-holonomic? 
\end{Que}

All the above results and questions easily translate into the language of subgroup generating functions for the number of free subgroups of finite index in $\mathbb{Z}_p*\mathbb{Z}_q$.  

\section{Subgroup conjugacy and hypermap isomorphism}

In order to count the conjugacy classes of index $n$ free subgroups in $\Delta^+$, we shall use Lemma~\ref{lemma:hypermaps-and-subgroups-2} and first compute the number $|\mathfrak{H}_{p,q}(n)|$ of isomorphism classes of $(p,q)$-hypermaps on $n$ darts.

We first recall the equations that hold for the species $H^*$ of labelled (not necessarily connected) $(p,q)$-hypermaps and the species $H$ of labelled connected $(p,q)$-hypermaps. As before, let $S_i$ be the species of permutations acting on the set of darts $D = [n]$, with cycles of length $i \geq 2$ only and without fixed points. Then 
\begin{equation}\label{eq:species-3}
H^* = S_p \times S_q, \,\, H^* = E(H).
\end{equation}

Let $C_p$ and $C_q$ be the species of $p$- and $q$-cycles, respectively, with their corresponding cycle indices
\begin{equation}\label{eq:species-4}
\mathcal{Z}_{C_p}(z_1, z_2, \dots, z_p) = \frac{1}{p}\, \sum_{d|p} \phi(d)\, z^{p/d}_d,\,\, \mathcal{Z}_{C_q}(z_1, z_2, \dots, z_q) = \frac{1}{q}\, \sum_{d|q} \phi(d)\, z^{q/d}_d,
\end{equation}
where $\phi(d)$ is the Euler totient function. 

The species $E$ of sets has cycle index
\begin{equation}\label{eq:species-5}
\mathcal{Z}_E(z_1, z_2, \dots) = \exp \left( \sum_{k\geq 1} \frac{z_k}{k} \right).
\end{equation}

Then, according to \cite[\S 1.4, Th\'{e}or\`{e}me 2 (c)]{Bergeron-et-al}, the cycle index of $S_p$ becomes
\begin{equation}\label{eq:species-6}
\mathcal{Z}_{S_p}(z_1, z_2, \dots) = \mathcal{Z}_E\left( \frac{1}{p}\, \sum_{d|p} \phi(d)\, z^{p/d}_d,\, \frac{1}{p}\, \sum_{d|p} \phi(d)\, z^{p/d}_{2d},\, \frac{1}{p}\, \sum_{d|p} \phi(d)\, z^{p/d}_{3d},\, \dots \right) = 
\end{equation}
\begin{equation}\label{eq:species-7}
 = \exp\left( \sum_{k\geq 1}\, \frac{1}{pk}\, \sum_{d|p}\, \phi(d)\, z^{p/d}_{kd} \right) = \exp\left( \sum_{k\geq 1}\, \sum_{d|p}\, \frac{1}{pk}\, \phi(d)\, z^{p/d}_{kd} \right) = 
\end{equation}
\begin{equation}\label{eq:species-8}
= \exp\left( \sum^\infty_{n=1}\, \sum_{d | (n, p)}\, \frac{d}{np}\, \phi(d)\, z^{p/d}_n \right) = \prod^\infty_{n=1}\, \exp\left( \frac{1}{np}\, \sum_{d | (n,p)} d\, \phi(d)\, z^{p/d}_n \right),
\end{equation}
where the last identity is obtained by setting $n = k d$, and $(n, p)$ denotes the greatest common divisor of $n$ and $p$.
An analogous expression holds for the cycle index $\mathcal{Z}_{S_q}(z_1, z_2, \dots)$ of the species $S_q$ (up to the substitution of $q$ in place of $p$).

Let us set 
\begin{equation}\label{eq:species-9}
P_n(z_n) = \exp\left( \frac{1}{np}\, \sum_{d|(n, p)} d\, \phi(d)\, z^{p/d}_n \right),\,\, Q_n(z_n) = \exp\left( \frac{1}{nq}\, \sum_{d|(n, q)} d\, \phi(d)\, z^{q/d}_n \right).
\end{equation}
Then
\begin{equation}\label{eq:species-10}
\mathcal{Z}_{S_p}(z_1, \dots, z_p) = \prod^\infty_{n=1} P_n(z_n),\,\, \mathcal{Z}_{S_q}(z_1, \dots, z_q) = \prod^\infty_{n=1} Q_n(z_n).
\end{equation}

By \eqref{eq:species-9} -- \eqref{eq:species-10}, the cycle indices $\mathcal{Z}_{S_p}$ and $\mathcal{Z}_{S_q}$ are \textit{separable}, c.f. \cite[Section A.3.3]{Petitot-Vidal}, which means that the cycle index $\mathcal{Z}_{H^*}$ can be expressed as
\begin{equation}\label{eq:species-11}
\mathcal{Z}_{H^*}(z_1, z_2, \dots) = \prod^\infty_{n=1} P_n(z_n) \odot Q_n(z_n).
\end{equation}
Given that two species $A$ and $B$ satisfy $A=E(B)$, by \cite[\S 1.4, Exercice 9]{Bergeron-et-al} we can obtain a formula for $\mathcal{Z}_B$ from $\mathcal{Z}_A$. In particular, the cycle index $\mathcal{Z}_H$ is 
\begin{equation}\label{eq:species-12}
\mathcal{Z}_H(z_1, z_2, \dots) = \sum^\infty_{n=1} \frac{\mu(n)}{n} \log \mathcal{Z}_{H^*}(z_n, z_{2n}, z_{3n}, \dots),
\end{equation}
where $\mu(n)$ is the M\"{o}bius function. According to \cite[\S 1.2, Th\'{e}or\`{e}me 8 (b)]{Bergeron-et-al}, the generating function $\widetilde{H}(z) = \sum_{n\geq 0} |\mathfrak{H}_{p,q}(n)| z^n$ satisfies
\begin{equation}\label{eq:gen-func-10}
\widetilde{H}(z) = \mathcal{Z}_{H}(z, z^2, z^3, \dots) = \sum^\infty_{n = 1} \frac{\mu(n)}{n}\, \log \mathcal{Z}_{H^*}(z^n, z^{2n}, z^{3n}, \dots) = 
\end{equation}
\begin{equation}\label{eq:gen-func-11}
= \sum^\infty_{n = 1} \frac{\mu(n)}{n}\, \sum^\infty_{k=1} \log(P_n\odot Q_n)(z_n)\vert_{z_n=z^{nk}}.  
\end{equation}

Thus we arrive at the following theorem:

\begin{Thm}\label{thm:p-q-hypermap-isom} Let $\widetilde{H}(z) = \sum^{\infty}_{n=0} | \mathfrak{H}_{p,q}(n)| z^n$ be the growth series for the number of isomorphism classes of $(p,q)$-hypermaps on $n$ darts.
\begin{itemize}
\item[(i)] Then the following formula holds:
\begin{equation*}
\widetilde{H}(z) = \sum^\infty_{n = 1} \frac{\mu(n)}{n}\, \sum^\infty_{k=1} \log(P_n\odot Q_n)(z_n)\vert_{z_n = z^{nk}},
\end{equation*}
with
\begin{equation*}
P_n(z_n) = \exp\left( \frac{1}{np}\, \sum_{d|(n, p)} d\, \phi(d)\, z^{p/d}_n \right),\,\, Q_n(z_n) = \exp\left( \frac{1}{nq}\, \sum_{d|(n, q)} d\, \phi(d)\, z^{q/d}_n \right).
\end{equation*}
\item[(ii)] The asymptotic expansion for its non-zero coefficients is
\begin{equation*}
[z^{\langle p,q \rangle k}] \widetilde{H}(z) \approx c_0 k^{c_1 k - \frac{1}{2}}\,\, e^{ c_2 k}, \mbox{ as } k \rightarrow \infty,
\end{equation*}
where $c_0, c_1, c_2 > 0$ are constants that depend only on $p$ and $q$.
\end{itemize}
\end{Thm}

\begin{proof}
Most of the proof has already been provided above. The remaining part is the asymptotic behaviour of $[z^n] \widetilde{H}(z)$, as $n \rightarrow \infty$. 

Since most of the $(p,q)$-hypermaps on $n$ darts are asymmetric (one can adapt the argument from \cite[Section 7.1]{Drmota-Nedela} in order to show this), we have that $| \mathfrak{H}_{p,q}(n)| \approx \frac{1}{n}\, | \mathfrak{H}^{r}_{p,q}(n)|$, as $n \rightarrow \infty$. Then we apply Theorem~\ref{thm:p-q-hypermaps-growth-transcendental} to obtain the required asymptotic formula:
\begin{equation*}
[z^{\langle p,q \rangle k}] \widetilde{H}(z) \approx \frac{ (2\pi)^{\frac{N-1}{2}}}{\prod_i \Gamma(a_i)} \,\, k^{(N-1)k - \frac{1}{2}}\,\, e^{ (1 + \log c - N) k}, \mbox{ as } k \rightarrow \infty,
\end{equation*}
where $N = 1 + \frac{(p-1)(q-1) - 1}{(p,q)}$, and the constants $c$ and $a_i$ satisfy $c^{(p,q)} = \langle p,q \rangle^{(p-1)(q-1)-1}$, $a_i = \frac{i (p,q)}{pq}$, with $i = 1,\dots, \frac{pq}{(p,q)}-1$, $p \nmid i$, $q \nmid i$.
\end{proof}

Again, we can reformulate our result in group-theoretic terms via Lemma~\ref{lemma:hypermaps-and-subgroups-2}.

\begin{Thm}\label{thm:conjugacy-classes}
Let $\Delta^+ = \mathbb{Z}_p * \mathbb{Z}_q$ be a free product of cyclic groups with $p, q > 1$, $pq\geq 6$. Then the conjugacy growth series $C_{f}(z) = \sum^{\infty}_{n=0} c_{f}(n) z^n$ for the numbers $c_{f}(n)$ of conjugacy classes of free subgroups of index $n$ in $\Delta^+$ coincides with $\widetilde{H}(z)$.
\end{Thm}

Below we consider several examples of the generating series $\widetilde{H}(z) = \sum^{\infty}_{n=0} | \mathfrak{H}_{p,q}(n)| z^n$ from Theorem~\ref{thm:p-q-hypermap-isom} for various values of $p$ and $q$. In order to facilitate computations with power series, we use the SAGE code from Appendix~III.

\begin{Ex}\label{example:isom-triangulations}
For the number of isomorphism classes of triangulations on $n$ darts we have $\widetilde{H}(z) = \sum^{\infty}_{n=0} |\mathfrak{H}_{2,3}(n)| z^n = 3\, z^6 + 11\, z^{12} + 81\, z^{18} + 1228\, z^{24} + 28174\, z^{30} + 843186\, z^{36} + 30551755\, z^{42} + 1291861997\, z^{48} + 62352938720\, z^{54} + 3381736322813\, z^{60} + \dots$. The coefficient sequence of $\widetilde{H}(z)$ has index A129114 in the OEIS \cite{OEIS}. 
\end{Ex}

\begin{Ex}\label{example:isom-quadrangulations}
For the number of isomorphism classes of quadrangulations on $n$ darts we have $\widetilde{H}(z) = \sum^{\infty}_{n=0} |\mathfrak{H}_{2,4}(n)| z^n = 2\, z^4 + 7\, z^8 + 36\, z^{12} + 365\, z^{16} + 5250\, z^{20} + 103801\, z^{24} + 2492164\, z^{28} + 70304018\, z^{32} + 2265110191\, z^{36} +  82013270998\, z^{40} + \dots$. The coefficient sequence of $\widetilde{H}(z)$ has index A292206 in the OEIS \cite{OEIS}. 
\end{Ex}

\begin{Ex}\label{example:isom-bicoloured-triangulations}
For the number of isomorphism classes of bi-coloured triangulations on $n$ darts we have $\widetilde{H}(z) = \sum^{\infty}_{n=0} |\mathfrak{H}_{3,3}(n)| z^n = 2\, z^3 + 3\, z^6 + 16\, z^9 + 133\, z^{12} + 1440\, z^{15} + 22076\, z^{18} + 401200\,  z^{21} + 8523946\, z^{24} + 206375088\, z^{27} + 5611089408\, z^{30} + \dots $. The coefficient sequence of $\widetilde{H}(z)$ has index A292207 in the OEIS \cite{OEIS}. 
\end{Ex}

\begin{Ex}\label{example:isom-fullerene-hypermaps}
For the number of isomorphism classes of fullerene hypermaps on $n$ darts we have $\widetilde{H}(z) = \sum^{\infty}_{n=0} |\mathfrak{H}_{5,6}(n)| z^n =  25267952661607723932\, z^{30} + 3242815920921585$ $11841901090716554032993726622535664\, z^{60} + 9893872851057565593775407161767990303$ $1448011824448933765884797305927298614859862448800\, z^{90} + \dots$.
\end{Ex}

\begin{Ex}\label{example:conjugacy-classes}
The conjugacy growth series for free subgroups in $\mathbb{Z}_2*\mathbb{Z}_3$, $\mathbb{Z}_2*\mathbb{Z}_4$, $\mathbb{Z}_3*\mathbb{Z}_3$ and $\mathbb{Z}_5*\mathbb{Z}_6$ are given in Examples \ref{example:isom-triangulations}, \ref{example:isom-quadrangulations}, \ref{example:isom-bicoloured-triangulations} and \ref{example:isom-fullerene-hypermaps}, respectively. An independent computation with GAP \cite{GAP} by issuing \texttt{LowIndexSubgroupsFPGroup} command gives matching results. 
\end{Ex}

To the best of our knowledge, the following questions remain at the moment unsettled. 

\begin{Que}
Is the growth series $\widetilde{H}(z)$ in Theorem~\ref{thm:p-q-hypermap-isom} algebraic or transcendental? 
\end{Que}

\begin{Que}
Is $\widetilde{H}(z)$ a solution to any differential-algebraic equation or system of such equations? Is there a recurrence relation of any sort that the coefficients of $\widetilde{H}(z)$ satisfy?
\end{Que}

\section{Other free products and their free subgroups}

As we mentioned in the introduction, the group $\Delta^+ = \mathbb{Z}_p*\mathbb{Z}_q$ is also known as the $(p,q,\infty)$-triangle group, a Fuchsian group with rich geometry. With the above technique we can also count the number of free subgroups of finite index and their conjugacy classes in $\Delta^+ = \mathbb{Z}_p*\mathbb{Z}$ and $\Delta^+ = \mathbb{Z}*\mathbb{Z}$, which are the $(p, \infty, \infty)$- and $(\infty, \infty, \infty)$-triangle groups, respectively.
These two groups have been fruitfully used by Breda d'Azevedo -- Menykh -- Nedela \cite{Breda-Mednykh-Nedela} and later on by Mednykh -- Nedela \cite{MN, MN1, MN2} for counting hypermaps: $\Delta^+ = \mathbb{Z}_2*\mathbb{Z}$ was used for counting maps without semiedges, and $\Delta^+ = \mathbb{Z}*\mathbb{Z}$ was used for counting general hypermaps (both rooted hypermaps and their isomorphism classes). 

Since most of the study has been already done in \cite{Breda-Mednykh-Nedela, MN, MN1, MN2}, we briefly formulate the necessary statements below without a proof. Where necessary, we give more details.

\subsection{Free subgroups of $\mathbb{Z}_p*\mathbb{Z}$}

If $\Delta^+ = \mathbb{Z}_p*\mathbb{Z}$, then its free subgroups of index $n$ are in bijection with the elements of $\mathfrak{H}^r_p(n)$, which is the set of rooted connected hypermaps on $n$ darts with $p$-gonal hyperedges. The conjugacy classes of free index $n$ subgroups in $\Delta^+$ are in bijection with the elements of $\mathfrak{H}_p(n)$, which is the set of isomorphism classes of rooted connected hypermaps on $n$ darts with $p$-gonal hyperedges. We shall call such hypermaps rooted, resp. unrooted, $(p,\infty)$-hypermaps.

Let us first define the species $H^\circ$ by the following equations
\begin{equation}\label{eq:species-13}
H^\circ = Z \cdot H^\prime, \,\, H^* = E(H), \mbox{ and } H^* = S_p\times S,
\end{equation}
analogous to equations \eqref{eq:species-1} -- \eqref{eq:species-2}, where $S$ is the species of all permutations, $H^*$ is the species of labelled, not necessarily connected, $(p,\infty)$-hypermaps, $H$ is the species of labelled connected $(p,\infty)$-hypermaps, and $H^\circ$ of rooted ones. Thus, we have that $\mathfrak{H}^r_p(n) = H^\circ[n]$.

Using the fact that 
\begin{equation}\label{eq:gen-func-12}
S_p(z) = \sum^\infty_{k=0} \frac{z^{pk}}{p^k k!}, \,\, S(z) = \frac{1}{1-z},
\end{equation}
we obtain 
\begin{equation}\label{eq:gen-func-13}
H^*(z) = (S_p \odot S)(z) = \sum^\infty_{k=0} \frac{(pk)!}{p^k k!} z^{pk}.
\end{equation}

Then, analogous to Section~\ref{s:hypermaps}, $H^*(z) = f(p^{p-1} z^p)$, where $f(x)$ is a hypergeometric function:
\begin{equation}\label{eq:gen-func-14}
f(x) = {}_{p}F_0\left[ \begin{array}{c}
\frac{1}{p}, \dots, \frac{p-1}{p}, 1\\
\ldots
\end{array}; x \right].
\end{equation}

We also have that $H^\circ(z) = p x \frac{f^\prime(x)}{f(x)} = p w(x)$, with $x = p^{p-1} z^p$, and $w(x)$ is a solution to an equation from the Riccati hierarchy with $N = p$, and $a_i = \frac{i}{p}$, for $i = 1, \dots, p$, c.f. Section~\ref{s:recurrence}.

\begin{Ex}\label{example:recurrence-maps}

If $p=2$, then $H^*(z)$ enumerates labelled oriented pre-maps (without semi-edges and, in general, disconnected) on $n$ darts or, equivalently, with $n/2$ edges, $n=1,2,\dots$, so that the odd coefficients of $H^*(z)$ vanish. We also have that $H^\circ(z) = \sum^\infty_{n=0} h(n) z^n$ enumerates oriented rooted maps on $n$ darts. All the odd coefficients of $H^\circ(z)$ vanish as well. Finally, $H^\circ(z) = 2 w(x)$, where $x = 2z^2$ and $w(x )$ satisfies the following Riccati equation (c.f. also \cite{AB}):
\begin{equation*}
w(x) = x^2 w^\prime(x) + \frac{3}{2} x w(x) + x w^2(x) + \frac{1}{2} x. 
\end{equation*}
Thus the coefficients of the series $w(x) = \sum^\infty_{n=0} w_n z^n$ satisfy $w_0 = 0$, $w_1 = \frac{1}{2}$ and
\begin{equation*}
w_{n+1} = \left(n + \frac{3}{2}\right) w_n + \sum^{n}_{i=0} w_i w_{n-i},\,\, n\geq 1.
\end{equation*}

This implies
\begin{equation*}
h(2n+2) = (2n+3) h(2n) + \sum^{n}_{i=0} h(2i) h(2n-2i),\,\, n\geq 1,
\end{equation*}
with $h(0) = 0$, $h(2) = 2$, and $h(2n+1) = 0$ for all natural $n\geq 0$. 

\end{Ex}

By analogy to Theorems \ref{thm:p-q-hypermaps-growth-transcendental} and \ref{thm:subgroup-growth-transcendental} the following holds. 

\begin{Thm}\label{thm:p-hypermaps-growth-transcendental}
The growth series $H^\circ(z) = \sum^{\infty}_{n=0} | \mathfrak{H}^r_p(n)| z^n$ is transcendental and the following asymptotic formula holds for its non-zero coefficients:
\begin{equation*}
[z^{pk}] H^\circ(z) \approx \frac{ (2\pi)^{\frac{p-1}{2}} \, p}{\prod^p_{i=1} \Gamma\left( \frac{i}{p} \right)} \,\, k^{(p-1)k + \frac{1}{2}}\,\, e^{ (p-1) (\log p - 1) k}, \mbox{ as } k \rightarrow \infty.
\end{equation*}
\end{Thm}

\begin{Ex}\label{example:maps}
For the case of oriented maps on $n$ darts without semi-edges (i.e. oriented connected $(2,\infty)$-hypermaps), we have that $H^\circ(z) = \sum^{\infty}_{n=0} | \mathfrak{H}^r_2(n)| z^n = 
2\, z^2 + 10\, z^4 + 74\, z^6 + 706\, z^8 + 8162\, z^{10} + 110410\, z^{12} + 1708394\, z^{14} + 29752066\, z^{16} + 576037442\, z^{18} + 12277827850\, z^{20} + \dots$ by using the SAGE code from Appendix~IV. The coefficient sequence of $H^\circ(z)$ has index A000698 in the OEIS \cite{OEIS}. Since $H^\circ(z)$ satisfies a classical Riccati equation, by Corollary \ref{cor:non-holonomic} this sequence is non-holonomic. This is also the $S(2,-3,1)$ sequence from \cite{Martin}.
\end{Ex}

\begin{Thm}\label{thm:subgroups-growth-transcendental-2}
Let $\Delta^+ = \mathbb{Z}_p * \mathbb{Z}$. Then the subgroup growth series $S_{f}(z) = \sum^{\infty}_{n=0} s_{f}(n) z^n$, for the number $s_{f}(n)$ of index $n$ free subgroups of $\Delta^+$, is equal to $H^\circ(z)$. 
\end{Thm}

\begin{Ex}
For $\Delta^+ = \mathbb{Z}_2*\mathbb{Z}$, the free subgroup growth series $S_f(z)$ is given in Example~\ref{example:maps}. 
\end{Ex}

The number of isomorphism classes of hypermaps or, equivalently, the number of conjugacy classes of free subgroups in $\Delta^+$, are given in the two theorems below.

\begin{Thm}\label{thm:p-hypermap-isom}
The growth series $\widetilde{H}(z) = \sum^{\infty}_{n=0} | \mathfrak{H}_{p}(n)| z^n$ is given by the formula
\begin{equation*}
\widetilde{H}(z) = \sum^\infty_{n = 1} \frac{\mu(n)}{n}\, \sum^\infty_{k=1} \log(P_n\odot Q_n)(z_n)\vert_{z_n=z^{nk}},
\end{equation*}
with
\begin{equation*}
P_n(z_n) = \exp\left( \frac{1}{np}\, \sum_{d|(n, p)} d\, \phi(d)\, z^{p/d}_n \right),\,\, Q_n(z_n) = \sum^\infty_{d=1} z^d_n = \frac{1}{1-z_n}.
\end{equation*}
Its non-zero coefficients have the asymptotic expansion
\begin{equation*}
[z^{pk}] \widetilde{H}(z) \approx \frac{ (2\pi)^{\frac{p-1}{2}} }{\prod^p_{i=1} \Gamma\left( \frac{i}{p} \right)} \,\, k^{(p-1)k - \frac{1}{2}}\,\, e^{ (p-1) (\log p - 1) k}, \mbox{ as } k \rightarrow \infty.
\end{equation*}
\end{Thm}

\begin{Ex}\label{example:maps-isomorphism}
For the case of isomorphism classes of oriented maps on $n$ darts (i.e. isomorphism classes of oriented connected $(2,\infty)$-hypermaps), we have that $\widetilde{H}(z) = \sum^{\infty}_{n=0} | \mathfrak{H}_2(n)| z^n = 2\, z^2 + 5\, z^4 + 20\, z^6 + 107\, z^8 + 870\, z^{10} + 9436\, z^{12} + 122840\, z^{14} + 1863359\, z^{16} + 32019826\, z^{18} + 613981447\, z^{20} + \dots$ by using the SAGE code from Appendix~IV. The coefficient sequence of $\widetilde{H}(z)$ has index A170946 in the OEIS \cite{OEIS}.
\end{Ex}

\begin{Thm}\label{thm:conjugacy-classes-2}
Let $\Delta^+ = \mathbb{Z}_p * \mathbb{Z}$ be a free product of cyclic groups with $p \geq 2$. Then the conjugacy growth series $C_{f}(z) = \sum^{\infty}_{n=0} c_{f}(n) z^n$, for the number $c_{f}(n)$ of conjugacy classes of index $n$ free subgroups in $\Delta^+$, coincides with $\widetilde{H}(z)$.
\end{Thm}

\begin{Ex}
For $\Delta^+ = \mathbb{Z}_2*\mathbb{Z}$, the conjugacy counting function $C_f(z)$ coincides with $\widetilde{H}(z)$ from Example~\ref{example:maps-isomorphism}. This can be verified by an independent computation using GAP's \texttt{LowIndexSubgroupsFPGroup} command \cite{GAP}.
\end{Ex}

\subsection{Free subgroups of $\mathbb{Z}*\mathbb{Z}$}

If $\Delta^+ = \mathbb{Z}*\mathbb{Z} \cong F_2$, the free group on two generators, then all of its index $n$ subgroups are free, and correspond bijectively to the elements of the set $\mathfrak{H}^r(n)$ of rooted hypermaps on $n$ darts. The conjugacy classes of index $n$ subgroups in $\Delta^+$ are in bijection with the elements of $\mathfrak{H}(n)$, which is the set of isomorphism classes of hypermaps on $n$ darts.  
In this case, the formulas for the number of index $n$ subgroups in $\Delta^+$ and their conjugacy classes are given by multiple authors including, respectively, Hall \cite{Hall}, for the former, and Liskovec \cite{L71}, Mednykh \cite{M-j-algebra}, for the latter. The advantage of a species theory approach is that the proof turns out to be much simpler and shorter. Initially, we have that
\begin{equation}\label{eq:hypermaps-1} 
H^*(z) = (S \odot S)(z) = \sum^\infty_{n=0} n! \cdot z^n,
\end{equation} 
and
\begin{equation}\label{eq:hypermaps-2}
H^\circ(z) = z \frac{d}{dz} \log H^*(z) = \sum^\infty_{n=0} s_n \cdot z^n,
\end{equation}
where $s_n$ is given by the recurrence relation (c.f. \cite{Hall})
\begin{equation}\label{eq:hypermaps-3}
s_n = n\cdot n! - \sum^{n-1}_{k=1} k!\cdot s_{n-k}, \mbox{ and } s_0 = 0,\,\, s_1 = 1.
\end{equation}
The sequence \eqref{eq:hypermaps-3} is known to be non-holonomic due to Klazar (c.f. discussion after \cite[Proposition 5.1]{Klazar}).

Also, 
\begin{equation}\label{eq:hypermaps-4}
\mathcal{Z}_{H^*}(z_1, z_2, \dots) = \prod^\infty_{n=1} \left( \sum^\infty_{i=0} i!\cdot n^i \cdot z^i_n \right)
\end{equation}
and
\begin{equation}\label{eq:hypermaps-5}
\widetilde{H}(z) = \sum^\infty_{k=1} \frac{\mu(k)}{k} \log \mathcal{Z}_{H^*}(z^k, z^{2k}, \dots) = \sum^\infty_{k=1} \frac{\mu(k)}{k} \sum^\infty_{n=1} \log\left( \sum^\infty_{i=0} i!\cdot n^i\cdot z^{ikn} \right).
\end{equation}

By using the fact that $\log \left( \sum^\infty_{i=0} i! z^i \right) = \sum^\infty_{i=1} \frac{s_i}{i} z^i$ (already used in formulas \eqref{eq:hypermaps-1}-\eqref{eq:hypermaps-2}), and substitution $z \rightarrow n\cdot z^{kn}$ we may simplify formula \eqref{eq:hypermaps-5} down to
\begin{equation}\label{eq:hypermaps-6}
\widetilde{H}(z) = \sum^\infty_{n=1} z^n \cdot \frac{1}{n} \sum_{i|n} s_i \sum_{m|\frac{n}{i}} \mu\left(  \frac{n}{i\, m}\right)\cdot m^{i+1} = 
\end{equation}
\begin{equation}\label{eq:hypermaps-7}
= \sum^\infty_{n=1} z^n \cdot \frac{1}{n} \sum_{\ell | n} s_{n/\ell}\, \varphi_{n/\ell+1}(\ell),
\end{equation}
where $\varphi_{m}(\ell) = \sum_{d|\ell} \mu\left( \frac{\ell}{d} \right)\, d^{m}$ is the Jordan totient function. The coefficient $[z^n]\widetilde{H}(z)$ in \eqref{eq:hypermaps-7} coincides with the one in  \cite[Proposition 2.1]{D}, \cite{L71} or \cite[Theorem 2]{M-j-algebra} and can be easily computed, for reasonable values of $n$, by using the SAGE code from Appendix~IV.

\section*{Appendix}

\subsection*{I. Counting free subgroups in $\mathbb{Z}_p*\mathbb{Z}_q$}

First we define the necessary parameters: $p$, $q$, and $n$, the index up to which we shall count the number of free subgroups in $\mathbb{Z}_p*\mathbb{Z}_q$.

\begin{verbatim}
# computing the number of free subgroups in Z_p*Z_q of index < n
# or: computing the number of rooted (p,q)-hypermaps on < n darts

# defining p, q (can be any positive integers such that p, q > 1, 
# p*q \geq 6)
p = 2; q = 3;

# defining n, which has to be a multiple of lcm(p,q)
n = 66;

# defining power series ring over \mathbb{Q}
R.<z> = PowerSeriesRing(QQ, default_prec = 2*n);
\end{verbatim}

Then we define the auxiliary function $f(x)$.

\begin{verbatim}
# defining f(z), such that H^\ast(z) = f(c*z^lcm(p,q))
f = 1; coeff = 1; 
for k in range(1,n):
    mult = 1;
    for i in range(1, lcm(p,q)):
        if (not(i%p==0) and not(i%q==0)):
            mult = mult*(k-1+i/lcm(p,q));
    coeff = coeff*mult;
    f = f + coeff*power(z,k)/factorial(k);
\end{verbatim}

We shall also need the constant $c$ for the substitution $x = c z^{\langle p,q \rangle}$.

\begin{verbatim}
# defining the constant c
c = power(lcm(p,q), (p*q - p - q)/gcd(p,q));
\end{verbatim}

Finally, $H^\circ(z)$ can be computed.

\begin{verbatim}
# computing H^\circ
Hcirc = lcm(p,q)*z*derivative(f)/f;
Hcirc.substitute(z = c*power(z,lcm(p,q))).truncate(n);
\end{verbatim}

\begin{verbatim}
>202903221120000*z^60 + 3366961243750*z^54 + 61999046400*z^48 +
1282031525*z^42 + 30220800*z^36 + 828250*z^30 + 27120*z^24 + 1105*z^18 +
60*z^12 + 5*z^6
\end{verbatim}

\subsection*{II. Finding a generalised Riccati equation for $w(x)$}

First we define $p$ and $q$, and then compute the number of terms in equation \eqref{eqn:riccati-3}. 

\begin{verbatim}
# defining p, q (can be any positive integers such that p, q > 1, 
# p*q \geq 6)
p = 2; q = 3;
\end{verbatim}

\begin{verbatim}
# computing N
N = 1 + (p*q - p - q)/gcd(p,q);
N;
>2
\end{verbatim}

Next we prepare all the ingredients for formula \eqref{eqn:riccati-3}, starting with the symmetric functions on the $a_i$'s, then computing the $w_i(x)$'s, and finally the equation itself. Below we use $Z$ instead of $x$ in the SAGE code throughout. 

\begin{verbatim}
# defining symmetric functions
e = SymmetricFunctions(QQ).elementary();
\end{verbatim}

\begin{verbatim}
# creating the list of a_i's
a_val = [];
for i in range(1,lcm(p,q)):
    if (not(i%p==0) and not(i%q==0)):
        a_val.append(i/lcm(p,q));
\end{verbatim}

\begin{verbatim}
# defining the symmetric function \sigma_k(a_1,\dots,a_N)
def sym_coeff(k):
    coeff = e([k]).expand(int(N), alphabet='a');
    return coeff(*a_val);
\end{verbatim}

\begin{verbatim}
# defining x and w as a function of x
var('Z'); w = function('w')(Z); w;
>Z
>w(Z)
\end{verbatim}

\begin{verbatim}
# defining w_i(x) for i = 1,\dots,N
w_lst = [1,w];
for i in range(1,N):
     w_lst.append(Z*w_lst[i].derivative(Z) + w*w_lst[i]);
\end{verbatim}

\begin{verbatim}
# defining the list of symmetric coefficients of the equation
coeff_lst = [sym_coeff(i) for i in range(N+1)];
\end{verbatim}

Now we compose the equation itself, and view the output.

\begin{verbatim}
# composing the equation
eqn = expand(w - Z*sum([a*b \\
						for a,b in zip(reversed(coeff_lst), w_lst)]));
# and viewing it
view(eqn);
\end{verbatim}

\begin{verbatim}
> - Z*w(Z)^2 - Z^2*D[0](w)(Z) - Z*w(Z) - 5/36*Z + w(Z)
\end{verbatim}

\subsection*{III. Counting free conjugacy classes in $\mathbb{Z}_p*\mathbb{Z}_q$}

We start by defining the parameters: $p$, $q$, and $n$, and the polynomial ring that we shall use. Although cycle indices are multivariate polynomials, we need only one variable $z$ to present the final result of computation as $\widetilde{H}(z)$.

\begin{verbatim}
# computing the number of conjugacy classes of 
# free subgroups in Z_p*Z_q of index < n
# or: computing the number of isomorphism classes 
# of (p,q)-hypermaps on < n darts

# defining p, q (can be any positive integers such that p, q > 1, 
# p*q \geq 6)
p = 2; q = 3;

# defining n, which has to be a multiple of lcm(p,q)
n = 66;

# defining power series ring over \mathbb{Q}
R.<z> = PowerSeriesRing(QQ, default_prec=2*n);
\end{verbatim}

Next we define the auxiliary functions $P_m(z_m)$ and $Q_m(z_m)$. Note that we use $z$ as a variable instead of $z_m$ (computing series in a single variable is usually faster in SAGE).

\begin{verbatim}
def P(m):
    sum = 0;
    for d in divisors(gcd(m,p)):
        sum = sum + d*euler_phi(d)*power(z, p//d);
    sum = sum/(m*p);
    return sum.exp(2*n);
\end{verbatim}

\begin{verbatim}
def Q(m):
    sum = 0;
    for d in divisors(gcd(m,q)):
        sum = sum + d*euler_phi(d)*power(z, q//d);
    sum = sum/(m*q);
    return sum.exp(2*n);
\end{verbatim}

Next we define the Hadamard product of $P_m(z_m)$ and $Q_m(z_m)$

\begin{verbatim}
def h_prod_PQ(m):
    prod = 0;
    P_coeff = P(m).dict();
    Q_coeff = Q(m).dict();
    for k in Set(P_coeff.keys()).intersection(Set(Q_coeff.keys())):
            prod = prod \\
            + power(z,k)*P_coeff[k]*Q_coeff[k]*factorial(k)*power(m,k);
    return prod;
\end{verbatim}

and its logarithm $\log(P_m\odot Q_m)(z_m)$ upon the substitution $z_m = z^{k m}$

\begin{verbatim}
def log_h_prod_PQ(m,k):
    return log(h_prod_PQ(m)).substitute(z=power(z,m*k));
\end{verbatim}

Finally, we define the general term of the double sum in Theorem~\ref{thm:p-q-hypermap-isom}

\begin{verbatim}
@parallel
def term(m,k):
    return moebius(k)/k*log_h_prod_PQ(m,k);
\end{verbatim}

in a way that allows parallel computing in order to speed up the computation.

The function $\widetilde{H}(z)$ is the double sum of the terms \texttt{term(m,k)} above. We define it, and compute its output.

\begin{verbatim}
# defining H_tilde(n):
def H_tilde(n):
    return sum([t[1] for t in list(term([(m,k) for m in range(1,n) \\
    for k in range(1,n)]))]).truncate(n);
# and computing it:
H_tilde(n);
\end{verbatim} 

\begin{verbatim}
>3381736322813*z^60 + 62352938720*z^54 + 1291861997*z^48 + 30551755*z^42
+ 843186*z^36 + 28174*z^30 + 1228*z^24 + 81*z^18 + 11*z^12 + 3*z^6
\end{verbatim}

\subsection*{IV. Free products $\mathbb{Z}_p*\mathbb{Z}$ and $\mathbb{Z}*\mathbb{Z}$}

In order to compute the subgroup growth series for the number of free subgroups having index $< n$ in $\mathbb{Z}_p*\mathbb{Z}$ we use essentially the same code as in Appendix~I.

\begin{verbatim}
# computing the number of free subgroups in Z_p*Z of index < n 
# or: computing the number of rooted (p,\infty)-hypermaps on < n darts
\end{verbatim}

\begin{verbatim}
# defining p (can be any integer \geq 2)
p = 2;
# defining n, which has to be a multiple of p
n = 42;
\end{verbatim}

\begin{verbatim}
# defining the power series ring
R.<z> = PowerSeriesRing(QQ, default_prec=2*n);
\end{verbatim}

The auxiliary function $f(x)$ is defined, where $x = p^{p-1}\, z^p$. Analogous to Appendix~I, we have that $H^*(z) = f(p^{p-1}\, z^p)$ and $H^\circ(z) = p^{p-1} z^p \frac{f'(z)}{f(z)}$.

\begin{verbatim}
# defining the auxiliary function f(x)
f = sum( [ factorial(p*k)/power(p, p*k)*power(z, k)/factorial(k) \\
											for k in range(n) ] );
\end{verbatim}

\begin{verbatim}
# computing H^\circ
Hcirc = p*z*derivative(f)/f;
Hcirc.substitute(z=power(p,p-1)*power(z,p)).truncate(n);
\end{verbatim}

\begin{verbatim}
>12444051435099603489508546*z^40 + 302625067295157128042954*z^38 +
7734158085942678174730*z^36 + 208256802758892355202*z^34 +
5925085744543837186*z^32 + 178676789473121834*z^30 +
5731249477826890*z^28 + 196316804255522*z^26 + 7213364729026*z^24 +
285764591114*z^22 + 12277827850*z^20 + 576037442*z^18 + 29752066*z^16 +
1708394*z^14 + 110410*z^12 + 8162*z^10 + 706*z^8 + 74*z^6 + 10*z^4 +
2*z^2
\end{verbatim}

Similarly we use our SAGE code from Appendix~III to compute the number of conjugacy classes of free subgroups in $\mathbb{Z}_p*\mathbb{Z}$. The only change that we perform is setting

\begin{verbatim}
def P(m):
    sum = 0;
    for d in divisors(gcd(m,p)):
        sum = sum + d*euler_phi(d)*power(z, p//d);
    sum = sum/(m*p);
    return sum.exp(2*n);
\end{verbatim}

\begin{verbatim}
def Q(m):
    return sum([power(z,k) for k in range(2*n)]);
\end{verbatim}

For instance, computing the number conjugacy classes of free subgroups in $\mathbb{Z}_p*\mathbb{Z}$ of index $< n$ with $p = 2$ and $n = 42$ produces the following output which is consistent with \cite[Table 1]{Breda-Mednykh-Nedela}.

\begin{verbatim}
H_tilde(n);
>311101285883236139915989*z^40 + 7963817561236130021156*z^38 +
214837724735760642773*z^36 + 6125200100394894738*z^34 +
185158932576089787*z^32 + 5955893472990664*z^30 + 204687564072918*z^28 +
7550660328494*z^26 + 300559406027*z^24 + 12989756316*z^22 +
613981447*z^20 + 32019826*z^18 + 1863359*z^16 + 122840*z^14 + 9436*z^12
+ 870*z^10 + 107*z^8 + 20*z^6 + 5*z^4 + 2*z^2
\end{verbatim}

Computing the number of subgroups in $\mathbb{Z}*\mathbb{Z} \cong F_2$ can be achieved by using Hall's recursive formula \cite{Hall}.  

In order to compute the number of conjugacy classes of subgroups in $\mathbb{Z}*\mathbb{Z} \cong F_2$, we again use the SAGE code from Appendix~III while setting $P_m(z_m)$ and $Q_m(z_m)$ to be 

\begin{verbatim}
def P(m):
    return sum([power(z,k) for k in range(2*n)]);
\end{verbatim}

and

\begin{verbatim}
def Q(m):
    return sum([power(z,k) for k in range(2*n)]);
\end{verbatim}

This produces the following output (as before we keep $n=42$):

\begin{verbatim}
>32614432894950329376183336374438251415525346629525*z^41 +
794947340495954068749307315565065240346321967624*z^40 +
19859819172105592803670953279393110577234174667*z^39 +
508851255700533446262838487197588054824465660*z^38 +
13380404248488655454173405967972741129540081*z^37 +
361334724689721826695385753211211533721996*z^36 +
10028311083012435161491928424499759688087*z^35 +
286257408185259477967964317085477976381*z^34 +
8411024399891941340574941137506814445*z^33 +
254611211538026870669260639344594486*z^32 +
7947648478276985324309955268981883*z^31 +
256066759683317825307629683049530*z^30 +
8524510378559020280174747825833*z^29 +
293538903330959739871704646496*z^28 + 10467749727936465019060859695*z^27
+ 387062034840260690448914883*z^26 + 14860597828358206165421269*z^25 +
593273186302387876241788*z^24 + 24667261441592452064563*z^23 +
1069983257460254131272*z^22 + 48509766592893402121*z^21 +
2303332664693034476*z^20 + 114794574818830735*z^19 +
6019770408287089*z^18 + 333041104402877*z^17 + 19496955286194*z^16 +
1211781910755*z^15 + 80257406982*z^14 + 5687955737*z^13 + 433460014*z^12
+ 35704007*z^11 + 3202839*z^10 + 314493*z^9 + 34470*z^8 + 4163*z^7 +
624*z^6 + 97*z^5 + 26*z^4 + 7*z^3 + 3*z^2 + z
\end{verbatim}

The coefficient sequence of the above series has index A057005 in the OEIS \cite{OEIS}.

\vspace*{0.5in}

\begin{table}[!hb]
\centering
\begin{tabular}{l@{\hskip 2in}l}
\begin{tabular}[l]{@{}l@{}} \it Laura Ciobanu\\ \it School of Mathematical\\ \it and Computer Sciences\\ \it Heriot-Watt University\\ \it 6100 Main Street\\ \it Edinburgh EH14 4AS, UK\\ \it l.ciobanu@hw.ac.uk\end{tabular} &
\begin{tabular}[l]{@{}l@{}}\it Alexander Kolpakov\\ \it Institut de Math\'{e}matiques\\ \it Universit\'{e} de Neuch\^{a}tel\\ \it Rue Emile-Argand 11\\ \it CH-2000 Neuch\^{a}tel\\ \it Suisse/Switzerland\\ \it kolpakov.alexander@gmail.com\end{tabular} 
\end{tabular}
\end{table}


\begin{thebibliography}{99}

\bibitem{AB}{D. Arqu\`{e}s, J.-F. B\'{e}raud}, {``Rooted maps on orientable surfaces, Riccati's equation and continued fractions''}, Discrete Math. \textbf{215}, 1--12 (2000).  

\bibitem{Bender}{E.A. Bender}, {``An asymptotic expansion for the coefficients of some formal power series''}, J. London Math. Soc. \textbf{9}, 451--458 (1975).

\bibitem{Bergeron-et-al}{F. Bergeron, G. Labelle, P. Leroux}, {``Th\'{e}orie des esp\`{e}ces et combinatoire des structures arborescentes''}, Publications du LaCIM, Universit\'{e} du Qu\'{e}bec \`{a} Montr\'{e}al, 1994.

\bibitem{Breda-Mednykh-Nedela}{A. Breda, A. Mednykh, R. Nedela}, {``Enumeration of maps regardless of genus. Geometric approach''}, Discrete Math. \textbf{310}, 1184--1203 (2010).

\bibitem{D}{M. Deryagina}, {``On the enumeration of hypermaps which are self-equivalent with respect to reversing the colors of vertices''}, in: Combinatorics and Graph Theory. Part V, Zap. Nauchn. Sem. POMI, 446, POMI, St. Petersburg, 2016, 31--39.

\bibitem{DLMF} {``Digital Library of Mathematical Functions''}, available on-line at \texttt{http://dlmf.nist.gov/}

\bibitem{Drmota-Nedela}{M. Drmota, R. Nedela}, {``Asymptotic enumeration of reversible maps regardless of genus''}, Ars Mathematica Contemporanea \textbf{5} (1), 77--97 (2012).

\bibitem{Flajolet-Sedgewick}{P. Flajolet, R. Sedgewick}, {``Analytic combinatorics''}, Cambridge University Press: Cambridge, 2009.

\bibitem{GAP} {``GAP - Groups, Algorithms, Programming - a System for Computational Discrete Algebra''}, available on-line at \texttt{http://www.gap-system.org/}

\bibitem{GIR}{C. Godsil, W. Imrich, R. Razen},  {``On the number of subgroups of given index in the modular group''}, Monatshefte f\"ur Mathematik \textbf{87},  273--280 (1979). 

\bibitem{Goulden-Jackson}{I.P. Goulden, D.M. Jackson}, {``The KP hierarchy, branched covers, and triangulations''}, Adv. Math. \textbf{219} (3), 932--951 (2008).

\bibitem{GdL}{J. Grabowski, J. de Lucas}, {``Mixed superposition rules and the Riccati hierarchy''}, J. Differential Equations \textbf{254} (1), 179--198 (2013).

\bibitem{Hall}{M. Hall, Jr.}, {``Subgroups of finite index in free groups''}, Canad. J. Math. \textbf{1}, 187--190 (1949).

\bibitem{Harris-Sibuya}{W.A. Harris, Y. Sibuya}, {``The reciprocals of solutions of linear ordinary differential equations''}, Adv. Math. \textbf{58} (2), 119--132 (1985).

\bibitem{Ch-thesis}{C. Hillar}, {``Solving polynomial systems with special structure''}, Ph.D. Thesis, UC Berkeley, 2005; available \href{http://www.msri.org/people/members/chillar/files/cjhthesis.pdf}{on-line}

\bibitem{Imrich}{W. Imrich}, {``On the number of subgroups of given index in $SL_2(\mathbb{Z})$''}. Archiv der Mathematik \textbf{31}  (1),  224--231 (1978).

\bibitem{Jones-Singerman}{G. A. Jones, D. Singerman}, {``Theory of maps on orientable surfaces''}, Proc. London Math. Soc. \textbf{37} (3), 273--307  (1978).

\bibitem{Joyal}{A. Joyal}, {``Une th\'{e}orie combinatoire des s\'{e}ries formelles''}, Adv. Math. \textbf{42} (1), 1--82 (1981).

\bibitem{Klazar}{M. Klazar}, {``Irreducible and Connected Permutations''}, IUUK-CE-ITI pre-print series (2003); available \href{http://iti.mff.cuni.cz/series/2003.html}{on-line}

\bibitem{L71}{V.A. Liskovec},{``On the enumeration of subgroups of a free group''}, Dokl. Akad. Nauk BSSR \textbf{15}, 6--9 (1971).

\bibitem{LS} {A. Lubotzky, D. Segal}, {``Subgroup Growth''}, Progress in Mathematics \textbf{212}, Birkh\"{a}user Verlag: Basel, 2003.

\bibitem{Martin} {R.J. Martin, M.J. Kearney}, {``An exactly solvable self-convolutive recurrence''}, Aequationes Math. \textbf{80} (3), 291--318 (2010).

\bibitem{M-j-algebra}{A. Mednykh}, {``Counting conjugacy classes of subgroups in a finitely generated group''}, J. Algebra \textbf{320} (6), 2209--2217 (2008).

\bibitem{MN}{A. Mednykh, R. Nedela}, {``Enumeration of unrooted hypermaps''}, Electron. Notes Discrete Math. \textbf{28},  207--214 (2007).

\bibitem{MN1}{A. Mednykh, R. Nedela}, {``Enumeration of unrooted maps of a given genus''}, J.~Combin. Theory, Ser.~B \textbf{96} (5), 706--729 (2006).

\bibitem{MN2}{A. Mednykh, R. Nedela}, {``Enumeration of unrooted hypermaps of a given genus''}, Discrete Math. \textbf{310} (3), 518--526 (2010).

\bibitem{MSP1}{T.W. M\"{u}ller, J.-C. Schlage-Puchta}, {``Classification and statistics of finite index subgroups in free products''}, Adv. Math. \textbf{188} (1), 1--50 (2004).

\bibitem{MSP2}{T.W. M\"{u}ller, J.-C. Schlage-Puchta}, {``Character theory of symmetric groups, subgroup growth of Fuchsian groups, and random walks''}, Adv. Math. \textbf{213} (2), 919--982 (2007).

\bibitem{MSP3}{T.W. M\"{u}ller, J.-C. Schlage-Puchta}, {``Statistics of isomorphism types in free products''}, Adv. Math. \textbf{224} (2), 707--730 (2010). 

\bibitem{NeBook}{R. Nedela}, {``Maps, Hypermaps and Related Topics''}, available \href{http://www.savbb.sk/~nedela/CMbook.pdf}{on-line}

\bibitem{Odlyzhko}{A.M. Odlyzko}, {``Asymptotic enumeration methods''}, in: R. L. Graham, M. Gr\"{o}tschel, and L. Lov\'{a}sz (eds.), Handbook of Combinatorics, vol. 2, Elsevier, 1995, 1063--1229.

\bibitem{Okounkov}{A. Okounkov}, {``Toda equations for Hurwitz numbers''}, Math. Res. Lett., \textbf{7}, 447--453 (2000). 

\bibitem{Petitot-Vidal}{M. Petitot, S. Vidal}, {``Counting Rooted and Unrooted Triangular Maps''}, J. Nonlin. Systems Applications \textbf{1} (2), 51--57 (2010).

\bibitem{Zeilberger-et-al}{M. Petkovsek, H.S. Wilf, D. Zeilberger}, {``A = B (with foreword by D.E. Knuth)''}, A.K. Peters: Wellesley MA, 1996.

\bibitem{OEIS}{N.J.A. Sloane et al.}, {``The On-line Encyclopaedia of Integer Sequences''}, \texttt{http://oeis.org}

\bibitem{Stothers-modular}{W.W. Stothers}, {``The number of subgroups of given index in the modular group''}, Proc.~Roy.~Soc.~Edinburgh, \textbf{78a}, 105--112 (1978).

\bibitem{Stothers}{W.W. Stothers}, {``Free Subgroups of Free Products of Cyclic Groups''}, Math. Comput. \textbf{32} (144), 1274--1280 (1977).

\bibitem{Vidal}{S. Vidal}, {``Sur la classification et le denombrement des sous-groupes du groupe modulaire et de leurs classes de conjugaison''}, \texttt{arXiv:math/0702223}. 

\bibitem{Yao}{W. Yao}, {``\"{U}ber die Anzahl der Untergruppen vom gegebenen Index in freien Produkten endlicher zyklischer Gruppen''}, Ph.D. thesis, 
Universit\"at Graz, 2000.
\end{thebibliography}
\end{document}